\documentclass[12pt,twoside]{amsart}

\usepackage[a4paper,width=150mm,top=25mm,bottom=25mm,bindingoffset=6mm]{geometry}

\usepackage{amsmath}
\usepackage{amsfonts}
\usepackage{graphicx}
\usepackage{amssymb}
\usepackage[all]{xy}
\usepackage{mathtools}
\usepackage{amssymb}
\usepackage[usenames,dvipsnames]{color}
\usepackage{latexsym}
\usepackage{xspace,enumerate,color}
\usepackage{bbm}
\usepackage{hyperref}
\usepackage{tikz}
\usetikzlibrary{arrows}
\usetikzlibrary{intersections,through,backgrounds}
\usetikzlibrary{calc}
\usetikzlibrary{decorations.markings}

\title{The Free Tangent Structure}
\author{Poon Leung, Department of Mathematics, Macquarie University}

\newtheorem{theorem}{Theorem}[section]

\newtheorem{definition}[theorem]{Definition}

\newtheorem{proposition}[theorem]{Proposition}

\newenvironment{nota}[1][Notation]{\begin{trivlist}
\item[\hskip \labelsep {\bfseries #1}]}{\end{trivlist}}

\newtheorem{example}[theorem]{Example}

\newenvironment{remark}[1][Remark]{\begin{trivlist}
\item[\hskip \labelsep {\bfseries #1}]}{\end{trivlist}}

\newcommand{\pgfextractangle}[3]{%
\pgfmathanglebetweenpoints{\pgfpointanchor{#2}{center}}
{\pgfpointanchor{#3}{center}}
\global\let#1\pgfmathresult
}   
\newcommand{\head}[1]{\textnormal{\textbf{#1}}}
\newcommand{\csize}{0.3cm}
\newcommand{\weil}{\ensuremath{\mathbf{Weil}}\xspace}
\newcommand{\alg}{\ensuremath{\mathbf{Alg}}\xspace}
\newcommand{\augalg}{\ensuremath{\mathbf{AugAlg}}\xspace}
\newcommand{\kmod}{\ensuremath{k\textnormal{-}\mathbf{Mod}}\xspace}
\newcommand{\vect}{\ensuremath{\mathbf{Vect}}\xspace}
\newcommand{\gph}{\ensuremath{\mathbf{Gph}}\xspace}
\newcommand{\peil}{\ensuremath{\mathbf{Weil_1}}\xspace}

\newcommand{\blank}{\underline{\hspace{0.3cm}}}
\newcommand{\ra}{\rightarrow}
\newcommand{\Ra}{\Rightarrow}
\renewcommand{\epsilon}{\varepsilon}

\newcommand{\cm}{\ensuremath{\mathcal M}\xspace}
\newcommand{\cc}{\ensuremath{\mathcal C}\xspace}
\newcommand{\bn}{\ensuremath{\mathbb N}\xspace}
\newcommand{\bz}{\ensuremath{\mathbb Z}\xspace}
\newcommand{\bt}{\ensuremath{\mathbb T}\xspace}
\newcommand{\btwo}{\ensuremath{\mathbbm 2}\xspace}

\newcommand{\br}{\ensuremath{\mathbb R}\xspace}
\newcommand{\cd}{\ensuremath{\mathcal D}\xspace}
\newcommand{\ca}{\ensuremath{\mathcal A}\xspace}
\newcommand{\cg}{\ensuremath{\mathcal G}\xspace}
\newcommand{\bg}{\ensuremath{\mathbb G}\xspace}
\newcommand{\en}{\ensuremath{\super (\cm )}\xspace}
\newcommand{\pbc}{\mbox{\LARGE{$\lrcorner$}}}

\DeclareMathOperator{\ind}{Ind}
\DeclareMathOperator{\cl}{Cl}

\DeclareMathOperator{\maxx}{max}
\DeclareMathOperator{\super}{End}
\DeclareMathOperator{\spec}{Spec}

\begin{document}

\maketitle

\begin{center}
\textbf{Abstract}
\end{center}

At the heart of differential geometry is the construction of the tangent bundle of a manifold. There are various abstractions of this construction, and this paper seeks to compare two of them: Synthetic Differential Geometry (SDG) and Tangent Structures.

Tangent structure is defined via giving an underlying category \cm and a tangent functor $T$ along with a list of natural transformations satisfying a set of axioms, then detailing the behaviour of $T$ in the category \en. SDG on the other hand is defined through the use of Weil algebras.

The aim of this paper is to present a more precise relationship between the two approaches for describing tangent structures.

\tableofcontents

\section{Introduction}

The starting point for the notion of tangent structure is that given  a smooth manifold $M$, we can construct the \textit{tangent space} $TM$, which to each point $x\in M$ attaches the vector space $T_xM$ of all tangents to $M$ at $x$. The functoriality of this construction is used to capture the idea of differentiation of maps between more abstract spaces.

$T$ being a functor (moreover an endofunctor over the category under consideration) allows us to talk about \textit{tangent structure}; the ingredients required to give a notion of ``tangent space" to an arbitrary category. There is also a more specific, technical meaning of ``tangent structure" given in \cite{rosi} and \cite{cocr}.

On the other hand, synthetic differential geometry defines tangent spaces and related structures through the use of infinitesimals (given as the spectrum of corresponding Weil algebras) as in \cite{kock}. The resulting tangent functor is then representable and moreover takes the form $[D,-]$ (for a particular infinitesimal $D$).

There are, as we shall see, strong connections between these two seemingly different approaches. Furthermore, it will turn out that the tangent functor $T$ is closely related to a particular Weil algebra in a very meaningful way. We shall begin with a brief look at tangent structure, then discuss Weil algebras and some of the constructions possible. We will then demonstrate how (co)graphs surprisingly play an important role in this discussion.


More specifically, we will introduce a category we will call \peil (a full subcategory of \weil) and detail a process for constructing any morphism of this category using a collection of generating morphisms (through the use of graphs). We will conclude by showing \textbf{Theorem \ref{bigtheorem}}, that to give a tangent structure (in the sense of \cite{cocr}) over a category $\cm$ is to give a functor
\begin{displaymath}
F\colon \peil\ra [\cm ,\cm ]
\end{displaymath}
satisfying certain axioms.

One final observation we will make is that we can in fact remove the requirement of the codomain of $F$ needing to be an endofunctor category $[\cm ,\cm ]$, and instead replace it with an arbitrary monoidal category $(\cg,\square,\textbf{1})$. This then more clearly exhibits \peil as what one might call the ``initial" tangent structure.

\section{Tangent Structure}\label{ts}

Tangent structure is defined by Rosick\'y \cite{rosi} using (internal) bundles of abelian groups, but we will be following the more general definition of Cockett-Cruttwell \cite{cocr} using (internal) bundles of commutative monoids. More explicitly, this requires that the tangent bundle $TM$ sitting over a smooth manifold $M$ is a commutative monoid, referred to as an \textit{additive bundle}.

\subsection{Internal commutative monoid}

\begin{definition} Given a category \cc, a {\em commutative monoid} in \cc consists of
\begin{quote}
\item[$\bullet$] An object $C$ such that finite powers of $C$ exist (the terminal object we shall call $t$);
\item[$\bullet$] A pair of maps $\eta \colon t\ra C$ and $\mu \colon C\times C\ra C$ such that the following diagrams commute
\vspace{-0.1cm}
\small{
\begin{equation*}
\hspace{-2cm} \xymatrix@C+1em@R+1em{
C\times (C\times C) \ar[r]^\alpha \ar[d]_{1\times\mu} & (C\times C)\times C \ar[r]^-{\mu\times 1} & C\times C \ar[d]^\mu & t\times C \ar[r]^{\eta\times 1} \ar[dr]_{\cong} & C\times C \ar[d]^\mu & C\times t \ar[l]_{1\times\eta} \ar[dl]^{\cong} \\
C\times C \ar[rr]_\mu & & C & & C & 
}   
\end{equation*}
}
and $\mu$ agrees with the symmetry map 
\begin{displaymath}
s\colon C\times C\ra C\times C~,
\end{displaymath}
so that the diagram
\small{
\begin{equation*}
\xymatrix@C+1em@R+1em{
C\times C \ar[r]^{s} \ar[dr]_\mu & C\times C \ar[d]^\mu \\
& C
}
\end{equation*}
}
also commutes.
\end{quote}

\end{definition}

\begin{remark}
Often, commutative monoids are considered in categories with all finite products, but we shall not be assuming this.
\end{remark}

\subsection{Additive bundles}

\begin{definition} If $A$ is an object in a category \cm, then an {\em additive bundle over $A$} is a commutative monoid in the slice category $\cm/A$. Explicitly, this consists of
\begin{quote}
\item[$\bullet$] a map $p\colon X\ra A$ such that pullback powers of $p$ exist, the $n^{th}$ pullback power denoted by $X^{(n)}$ and projections $\pi_i\colon X^{(n)}\ra X$ for $i\in \{1,\dots,n\}$;
\item[$\bullet$] maps $+\colon X^{(2)}\ra X$ and $\eta\colon A\ra X$ with $p\circ += p\circ\pi_1=p\circ\pi_2$ and $p\circ \eta=id$ which are associative, commutative, and unital.
\end{quote}
\end{definition}

\begin{remark}
We will note here that the notation used in \cite{cocr} for the $n^{th}$ pullback power is instead $X_n$.
\end{remark}

\begin{definition} Suppose $p\colon X\ra A$ and $q\colon Y\ra B$ are additive bundles. An {\em additive bundle morphism} is a pair of maps $f\colon X\ra Y$ and $g\colon A\ra B$ such that the following diagrams commute.
\begin{equation*}
\xymatrix@C+1em@R+1em{
X \ar[r]^-{f} \ar[d]_-{p} & Y \ar[d]^-{q} & X^{(2)} \ar[r]^-{f\times f} \ar[d]_-{+} & Y^{(2)} \ar[d]^-{+} & A \ar[r]^-{g} \ar[d]_-{\eta} & B \ar[d]^-{\eta'} \\
A \ar[r]_-{g} & B & X\ar[r]_-{f} & Y & X\ar[r]_-{f} & Y
}   
\end{equation*}
\end{definition}

\subsection{Tangent structure (in the sense of \cite{cocr})}

\begin{definition}
Given a category \cm, a {\em tangent structure} $\mathbb{T}=(T,p,\eta,+,l,c)$ consists of
\begin{quote}
\item[$\bullet$] (\textbf{tangent functor}) a functor $T\colon \cm\ra \cm$ and a natural transformation $p\colon T\ra 1_{\cm}$ such that pullback powers $T^{(n)}$ of $p$ exist and the composites $T^m$ of $T$ preserve these pullbacks for all $m\in\mathbb{N}$;
\item[$\bullet$] (\textbf{tangent bundle}) natural transformations $+\colon T^{(2)}\Ra T$ and $\eta\colon 1_{\cm}\Ra T$ making $p\colon T\ra 1_{\cm}$ into an additive bundle;
\item[$\bullet$] (\textbf{vertical lift}) a natural transformation $l\colon T\Ra T^2$ such that
\begin{displaymath}
(l,\eta)\colon (p,+,\eta)\ra\big{(}Tp,T+,T\eta\big{)}
\end{displaymath}
is an additive bundle morphism;
\item[$\bullet$] (\textbf{canonical flip}) a natural transformation $c\colon T^2\ra T^2$ such that
\begin{displaymath}
(c,id_T)\colon \big{(}Tp,T+,T\eta\big{)}\ra\big{(}pT,+T,\eta T\big{)}
\end{displaymath}
is an additive bundle morphism;
\end{quote}
where the natural transformations $l$ and $c$ satisfy
\begin{quote}
\item[$\bullet$] (\textbf{coherence of $l$ and $c$}) $c^2=id$, $c\circ l=l$, and the following diagrams commute
\begin{equation*}
\xymatrix@C+1em@R+1em{
T \ar[r]^{l} \ar[d]_{l} & T^2 \ar[d]^{Tl} & T^3 \ar[r]^{Tc} \ar[d]_{cT} & T^3 \ar[r]^{cT} & T^3 \ar[d]^{Tc} & T^2 \ar[r]^{lT} \ar[d]_{c} & T^3 \ar[r]^{Tc} & T^3 \ar[d]^{cT} \\
T^2 \ar[r]_{lT} & T^3 & T^3 \ar[r]_{Tc} & T^3 \ar[r]_{cT} & T^3 & T^2 \ar[rr]_{Tl} & & T^3
}   \vspace{-1.05cm}
\end{equation*}
\hspace{14.5cm};
\vspace{0.5cm}
\item[$\bullet$] (\textbf{universality of vertical lift}) the following is an equaliser diagram
\begin{equation*}
\xymatrix@C+1em@R+1em{
T^{(2)} \ar[rr]^{ (T+) \circ (l \times_T \eta T)} & & T^2 \ar@<-.5ex>[rr]_{\eta\circ p\circ Tp} \ar@<.5ex>[rr]^{Tp} & & T~,
}
\end{equation*}

where $(T+) \circ (l \times_T \eta T)$ is the composite
\begin{equation*}
\xymatrix@C+1em@R+1em{
& T \ar[r]^l & T^2 & \\
T^{(2)}\ar[ur]^{\pi_1} \ar[dr]_{\pi_2} \ar@{-->}[rr] & & TT^{(2)} \ar[u]^{\pi_1} \ar[d]_{\pi_2} \ar[r]^{T+} & T^2 \\
& T\ar[r]_{\eta T} & T^2 &
}
\end{equation*}
\end{quote}
\end{definition}

\begin{remark}
We will note here that $l:T\ra T^2$ and $p:T\ra 1_\cm$ do not form a comonad. However, there is a canonical way to make $T$ a monad (detailed in \cite{cocr}).
\end{remark}

\subsection{Algebra preliminaries} \textcolor{white}{all work and no play makes jack a dull boy}

With this definition of tangent structure in mind, we will introduce the idea of Weil algebras. For the purposes of this paper, we will be using commutative, unital algebras defined over a commutative rig $k$; recall a \textit{rig} is a commutative monoid equipped with a (unital) multiplication.

In particular, if we take $k$ to be (the commutative ring) $\mathbb{Z}$, we will ultimately recover the abelian group bundles of \cite{rosi}, while (the commutative rig) $\mathbb{N}$ corresponds to the additive bundles of \cite{cocr}. Later on in our discussion, we will be interested in the rig $\btwo =\{ 0,1\}$, where multiplication is the usual one and addition is given by $\maxx$, in particular $1+1=1$.

\section{Weil Algebras}

\subsection{Definition and basic concepts} \textcolor{white}{all work and no play makes jack a dull boy}

\begin{definition}
A {\em Weil algebra} $B$ is an augmented algebra whose underlying $k$-module is finitely generated and free, for which all elements of the augmentation ideal are nilpotent.
\end{definition}

\begin{remark}
If $k$ is a field, then a Weil algebra is simply a finite dimensional local $k$-algebra with residue field $k$.
\end{remark}

A morphism between Weil algebras $B$ and $C$ is simply an augmented algebra homomorphism, i.e. an algebra map

\begin{displaymath}
f\colon B \ra C
\end{displaymath}
that is compatible with the augmentations, i.e. we have a commuting diagram
\begin{equation*}
  \xymatrix@C+1em@R+1em{
   B \ar[r]^-{f} \ar[d]_-{\epsilon_B} & C \ar[dl]^-{\epsilon_C}  \\
   k
  }   
 \end{equation*}

From here onwards, these augmented algebra homomorphisms will simply be referred to as maps.

We can now consider the category \weil with objects the Weil algebras and morphisms the augmented algebra homomorphisms so that it is a full subcategory of  \augalg$(=$\ensuremath{\mathbf{Alg}/k}\xspace, the category of augmented algebras).

We will note at this point that we will (soon) further restrict \weil to a full subcategory \peil in order to discuss tangent structure.

It is often convenient to give a Weil algebra $B$ via a presentation
\begin{displaymath}
 B=k[b_1,\dots,b_m]/Q_B ,
\end{displaymath}
where $Q_B$ is a list of polynomial terms in the generators $b_1,\dots,b_m$ by which to quotient out (i.e. terms that equal 0).

\begin{remark} All Weil algebras have such a presentation (finding it is another issue).
\end{remark}

\begin{example}\textcolor{white}{daddy needs to express some rage}

\begin{itemize}
\item[$\bullet$] $k[x]/x^2$ is the Weil algebra with $\lbrace 1,x\rbrace$ as a basis for the underlying $k$-module and equipped with the obvious multiplication, but with $x^2$ identified as 0.
\item[$\bullet$] $k[x]/x^3$ is the Weil algebra with $\lbrace 1,x,x^2\rbrace$ as a basis for the underlying $k$-module and equipped with the obvious multiplication, but with $x^3$ identified with 0.
\item[$\bullet$] $k[x,y]/x^2,y^2$ is the Weil algebra with $\lbrace 1,x,y,xy\rbrace$ as a basis for the underlying $k$-module and equipped with the obvious multiplication, but with $x^2$ and $y^2$ each identified with 0.
\end{itemize}
\end{example}

We also note the following:
\begin{itemize}
\item[$\bullet$] We shall always use presentations for which the augmentation $\epsilon\colon B \ra k$ sends each generator $b_i$ to 0.
\item[$\bullet$] Recall that for a linear map $h\colon X \ra Y$ between vector spaces, it suffices to define how $h$ acts on basis elements of $V$. Analogously, for an augmented algebra homomorphism $f\colon B \ra C$, it suffices to define how $f$ acts on generators (then check that it is suitably compatible with the relations).
\item[$\bullet$] For Weil algebras $A=k[a_1,\dots,a_m]/Q_A$ and $B=k[b_1,\dots,b_n]/Q_B$ and a map $f\colon A\ra B$, $f(a_i)$ is a polynomial in the generators $b_1,\dots,b_n$ with no constant term.
\end{itemize}

\subsection{Constructions with Weil algebras}\label{weil facts 1} \textcolor{white}{\tiny{all work and no play makes jack a dull boy}}

The results we will describe in \textbf{3.2} and \textbf{3.3} are true for well-behaved $k$, in particular if it is a field, $\bz$, $\bn$ or $\btwo$. We begin with the following facts:
\begin{itemize}
\item[$\bullet$] The category \augalg has all limits and colimits.
\item[$\bullet$] The forgetful functor $U\colon\augalg\ra\kmod$ preserves connected limits.
\item[$\bullet$] Coproducts in \augalg are given by $\otimes$.
\end{itemize}

We will first establish some facts about \weil in order to define the class of ``foundational pullbacks".

\subsubsection{Products}

\begin{proposition}
The category \weil is closed under finite products.
\end{proposition}

\begin{proof}
Since $k$ is a zero object, then it is the nullary product.

For arbitrary Weil algebras $A$ and $B$, begin by taking the pullback
\begin{equation*}
\xymatrix{
\ar @{} [dr] | \pbc A\times_k B \ar[r] \ar[d] & B\ar[d]^{\epsilon_b} \\
A \ar[r]_{\epsilon_A} & k
} 
\end{equation*}
in \kmod (or equivalently, the product in \augalg). Since both $A$ and $B$ are finitely generated, free and have nilpotent augmentation ideals, then the same is true of $A\times_k B$. Thus it is also a Weil algebra.
\end{proof}

\begin{proposition}
The forgetful functor $U\colon \weil\ra \kmod$ preserves and reflects pullbacks.
\end{proposition}

\begin{proof}
The category \augalg has all pullbacks, and these are preserved and reflected by the forgetful functor
\begin{displaymath}
U'\colon \augalg\ra\kmod
\end{displaymath}

Since \weil is a full subcategory of \augalg, then reflection of pullbacks follows immediately.

As for preservation of pullbacks, note that 
\begin{displaymath}
\{ k[x]/x^n~|~n\in\bn\}
\end{displaymath}
forms a strong generator for \augalg and also lies in \weil. It follows that the forgetful functor $U\colon \weil\ra\kmod$ preserves and reflects pullbacks.
\end{proof}

\subsubsection{Coproducts}

\begin{proposition} The category $\weil$ has all finite coproducts.
\end{proposition}

\begin{proof}
(Finite) coproducts in \alg are given by $\otimes$. Colimits in \augalg are as in \alg. It remains to show that \weil is closed in \augalg under $\otimes$. Clearly, since Weil algebras are finitely generated and free, then a finite coproduct of Weil algebras will also be finitely generated and free. The nilpotency of the augmentation ideal is almost immediate.
\end{proof}

\begin{remark}
let $A$ and $B$ be two arbitrary Weil algebras with presentations
\begin{align*}
A &=k[a_1,\dots,a_m]/Q_A \\
B &=k[b_1,\dots,b_n]/Q_B
\end{align*}

It can be readily shown that
\begin{quote}
\item[$\bullet$] The product $A\times B$ has presentation
\begin{displaymath}
A\times B = k[a_1,\dots,a_m,b_1,\dots,b_n]/ Q_A \cup Q_B \cup [ a_i b_j|\forall i,j ]
\end{displaymath}
\item[$\bullet$] The coproduct $A\otimes B$ has presentation
\begin{displaymath}
A\otimes B = k[ a_1,\dots,a_m,b_1,\dots,b_n]/ Q_A \cup Q_B
\end{displaymath}
\end{quote}
\end{remark}

Finally, let us define $W$ to be the Weil algebra $k[ x] /x^2$ (we will use this notation from here onwards), then the n$^{th}$ power and copower of $W$, denoted $W^n$ and $nW$ respectively, have representations
\begin{align*}
W^n &= k[ x_1,\dots,x_n] / \lbrace {x_i}{x_j}|\hspace{0.1cm}\forall i\leqslant j\rbrace \\
nW &= k[ x_1,\dots,x_n] / \lbrace {x_i}^2|\hspace{0.1cm}\forall i\rbrace
\end{align*}

\subsection{Further properties of Weil algebras}\label{weil facts 2} \textcolor{white}{\tiny{all work and no play makes jack a dull boy}}

We shall now detail some properties of the constructions above that will be important later on.

\begin{proposition}
For a finitely generated and free $k$-module $V$, the functor
\begin{displaymath}
V\otimes\blank\colon \kmod\ra\kmod
\end{displaymath}
preserves pullbacks.
\end{proposition}

\begin{proof}
Since $V$ is finite dimensional, let $\dim V =n$. Then $V\cong k^n$ and thus
\begin{displaymath}
V\otimes\blank\cong (\blank)^n\colon \kmod\ra\kmod,
\end{displaymath}
 and this functor preserves all limits.
\end{proof}

\begin{proposition}
Given any Weil algebra $A$, the functor 
\begin{displaymath}
A\otimes\underline{\hspace{0.3cm}}\colon \weil\ra\weil
\end{displaymath}
preserves any existing pullbacks (in particular, pullbacks over k).

\end{proposition}

\begin{proof}
Note the commuting diagram in \textbf{Cat}
\begin{equation*}
  \xymatrix@C+2em@R+2em{
   \weil \ar[r]^-{A\otimes\underline{\hspace{0.3cm}}} \ar[d]_-{U} & \weil \ar[d]^-{U}  \\
   \vect \ar[r]_-{A\otimes\underline{\hspace{0.3cm}}} & \vect
  }   
\end{equation*}

Then by starting with a pullback diagram in the top left instance of \weil\ and applying the previous facts about the forgetful functor $U$ and the $A\otimes\blank\colon \vect\ra\kmod$, then the result is immediate.
\end{proof}

\begin{remark} Since $\otimes$ is commutative, then for any two Weil algebras $A$ and $B$, the functor $A\otimes\underline{\hspace{0.3cm}}\otimes B$ also preserves pullbacks.
\end{remark}

As such, for Weil algebras $A$, $B$ and $C$, we will refer to precisely all pullbacks of the form
\begin{equation*}
\xymatrix{
\ar @{} [dr] | \pbc A\otimes (B\times C) \ar[r]^-{A\otimes \pi_B} \ar[d]_{A\otimes \pi_C} & A\otimes B \ar[d]^{A\otimes\epsilon_B} \\
A\otimes C \ar[r]_{A\otimes\epsilon_C} & A
} 
\end{equation*}
as ``foundational pullbacks".

\begin{remark}
For the case of $k$ being a field, $k$-modules are vector spaces and are automatically free, and \weil in fact has all connected limits.
\end{remark}

\subsection{Tangent structure and Weil algebras}\textcolor{white}{\tiny{all work and no play makes jack a dull boy}}

The tangent functor $T$ is closely related to the Weil algebra $W=k\lbrace x\rbrace /x^2$. In synthetic differential geometry (i.e. in the sense of \cite{kock}), $T$ is the representable functor $(\underline{\hspace{0.3cm}})^D$, where $D=\spec (W)$.

Here, we will begin to describe a different relationship between \weil and tangent structure. Regard coproduct $\otimes$ as a monoidal operation on \weil (with unit $k$).

\begin{proposition}
The (endo)functor
\begin{displaymath}
W\otimes\underline{\hspace{0.3cm}}\colon\weil\ra\weil
\end{displaymath}
can be used to define a tangent structure on \weil.
\end{proposition}

\begin{proof}
With $T=W\otimes\underline{\hspace{0.3cm}}$, we first give the natural transformations required in order to have a tangent structure on \weil. The names for the morphisms used below will be deliberately chosen to coincide with those of tangent structure.
\vspace{0.1cm}
\begin{center}
\begin{tabular}{ccc}
\hline
& \head{Natural transformation} & \head{Explanation}\\
\hline
Projection & $\epsilon_W\otimes\underline{\hspace{0.3cm}}\colon T\Ra id_\weil$ & $\epsilon_W\colon W\ra k$ is the augmentation for $W$\\
Addition & $+\otimes\underline{\hspace{0.3cm}}\colon T^{(2)} \Ra T$ & $T^{(2)}$ is the functor $W^2\otimes\underline{\hspace{0.3cm}}$,\\
& & $+\colon W^2\ra W$; $x_1,x_2\mapsto x$\\
Unit & $\eta\otimes\underline{\hspace{0.3cm}}\colon id_\weil\Ra T$ & $\eta\colon k\ra W$ is the (multiplicative) unit for $W$\\
Vertical lift & $l\otimes\underline{\hspace{0.3cm}}\colon T\Ra T^2$ & $T^2=T\circ T$ is the functor $2W\otimes\underline{\hspace{0.3cm}}$\\
& & $l\colon W\ra 2W$; $x\mapsto x_1x_2$\\
Canonical flip & $c\otimes\underline{\hspace{0.3cm}}\colon T^2\Ra T^2$ & $c\colon 2W\ra 2W$; $x_i\mapsto x_{3-i}$, for $i=1,2$\\
\hline
\end{tabular}
\end{center}
With these choices of natural transformations as well as the facts established in \textbf{3.2} and \textbf{3.3}, it is a very routine exercise to verify that this does in fact define a tangent structure on \weil; in particular, this uses the fact that

\begin{equation*}
  \xymatrix@C+1em@R+1em{
W^2 \ar[rrr]^{(W\otimes +)\circ (l\times_W (\eta_W \otimes W))} & & & 2W \ar@<-.5ex>[rr]_{\eta_W\circ(\epsilon_W\otimes\epsilon_W)} \ar@<.5ex>[rr]^{W\otimes\epsilon_W} & & W
}
\end{equation*}
is an equaliser in \peil. Here, the map $(W\otimes +)\circ (l\times_W (\eta_W \otimes W))$, which we will denote as $v$, is given by
\begin{align*}
k[x_1,x_2]/x_1^2,x_2^2,x_1x_2 &\ra k[y_1,y_2]/y_1^2,y_2^2\\
x_1&\mapsto y_1y_2 \\
x_2&\mapsto y_2~.
\end{align*}
The map $W\otimes\epsilon_W:k[y_1,y_2]/y_1^2,y_2^2\ra k[z]/z^2$ sends $y_1$ to $z$ and $y_2$ to $0$, and $\eta_W\circ(\epsilon_W\otimes\epsilon_W):k[y_1,y_2]/y_1^2,y_2^2\ra k[z]/z^2$ sends both $y_1$ and $y_2$ to $0$.

\end{proof}
It is important to note that this tangent structure on \weil arises from the object $W$, its (finite product) powers $W^n$ and tensors of these. With this in mind, it makes sense to take a full subcategory \peil of \weil whose objects are given by the closure of the set $\{W^n\}_{n\in \bn}$ under finite tensors.

We shall first discuss this subcategory \peil at the level of objects and give a convenient way to classify them. After that, we shall show that the morphisms of this subcategory can be constructed from those described in the table above.

Given the presentations for products and coproducts described in \textbf{3.2.2}, then clearly any object of \peil will have a presentation of the form
\begin{displaymath}
C=k\lbrace c_1,...,c_n\rbrace / \lbrace c_ic_j \hspace{0.1cm} | \hspace{0.1cm} c_i\sim c_j \hspace{0.1cm} \rbrace
\end{displaymath}
for some symmetric reflexive relation $\sim$ (although not all such symmetric reflexive relations will yield an object of \peil). We will suppress the reflexive property of such relations.

However, a symmetric reflexive relation (in particular on a finite set) can be represented as a graph. More importantly, the use of graphs allows us to say explicitly which Weil algebras we wish to include in the subcategory \peil, and furthermore will later allow us to describe the morphisms of \peil.

\section{Graphs}

\subsection{Graph fundamentals} \textcolor{white}{all work and no play makes jack a dull boy}

Let us begin by defining some basic concepts relating to (finite simple) graphs that we will need to use.

\begin{definition}
A {\em graph} $G$ is a pair of sets $(V,E)$, with $V$ a finite set of ``vertices" of $G$, and $E$ a set of unordered pairs of distinct vertices, called the ``edges" of $G$.
\end{definition}

\begin{example}
$G=\Big{(}\lbrace 1,2,3,4,5,6\rbrace,\lbrace (1,2),(1,3),(1,6),(2,3),(4,5)\rbrace\Big{)}$ is the graph
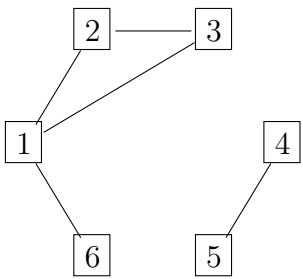
\begin{center}
\begin{tikzpicture}
\coordinate (X1) at (-0.2,0) {};
\draw (X1) +(-7pt,-7pt) rectangle +(6.5pt,8.5pt);
\node at (X1) {\textnormal{1}};

\coordinate (X2) at (0.7,1.5) {};
\draw (X2) +(-7pt,-7pt) rectangle +(6.5pt,8.5pt);
\node at (X2) {\textnormal{2}};

\coordinate (X3) at (2.3,1.5) {};
\draw (X3) +(-7pt,-7pt) rectangle +(6.5pt,8.5pt);
\node at (X3) {\textnormal{3}};

\coordinate (X4) at (3.2,0) {};
\draw (X4) +(-7pt,-7pt) rectangle +(6.5pt,8.5pt);
\node at (X4) {\textnormal{4}};

\coordinate (X5) at (2.3,-1.5) {};
\draw (X5) +(-7pt,-7pt) rectangle +(6.5pt,8.5pt);
\node at (X5) {\textnormal{5}};

\coordinate (X6) at (0.7,-1.5) {};
\draw (X6) +(-7pt,-7pt) rectangle +(6.5pt,8.5pt);
\node at (X6) {\textnormal{6}};

\draw[-,>=triangle 45] ($(X1)!\csize!(X2)$) -- ($(X2)!\csize!(X1)$);
\draw[-,>=triangle 45] ($(X1)!\csize!(X3)$) -- ($(X3)!\csize!(X1)$);
\draw[-,>=triangle 45] ($(X2)!\csize!(X3)$) -- ($(X3)!\csize!(X2)$);
\draw[-,>=triangle 45] ($(X1)!\csize!(X6)$) -- ($(X6)!\csize!(X1)$);
\draw[-,>=triangle 45] ($(X4)!\csize!(X5)$) -- ($(X5)!\csize!(X4)$);
\end{tikzpicture}
\end{center}
\end{example}

\begin{remark} For the purposes of our calculations, we will only need simple finite graphs, so this is what we will mean when we say ``graph".
\end{remark}

For a non-empty graph $G=(V,E)$, we will say it is \textit{connected} if for any two distinct vertices $u$ and $v$, there exist $v_1,\dots,v_s$ with $(v_i,v_{i+1})\in E$ for each $i$, with $v_1=u$ and $v_s=v$. Furthermore, we will say that $G$ is \textit{discrete} if the edge set $E$ is empty.

\begin{remark} Notice that under this convention, the one-point graph is regarded as being both connected and discrete whilst the empty graph is neither.
\end{remark}

\begin{definition}
Given a graph $G=(V,E)$, the {\em complement} of $G$ is the graph $G^C=(V,E^C)$, where for any two distinct $u,v\in V$,
\begin{displaymath}
(u,v)\in E\Leftrightarrow (u,v)\notin E^C .
\end{displaymath}
\end{definition}

We now define two important binary operations on graphs. Let $G_1=(V_1,E_1)$ and $G_2=(V_2,E_2)$ be two graphs; without loss of generality, assume that $V_1$ and $V_2$ are disjoint sets, hence $E_1$ and $E_2$ are also disjoint.

\begin{definition}
The {\em disjoint union} of $G_1$ and $G_2$, denoted as $G_1\otimes G_2$, is the graph
\begin{displaymath}
G_1\otimes G_2=(V_1\cup V_2,E_1\cup E_2)
\end{displaymath}
\end{definition}

Or, put simply, it is the graph given by simply placing $G_1$ adjacent to $G_2$ without adding or removing any edges.

\begin{definition}
The {\em graph join} of $G_1$ and $G_2$, denoted $G_1\times G_2$, is the graph
\begin{displaymath}
G_1\times G_2=(V_1\cup V_2,\widetilde{E})
\end{displaymath}
where $\widetilde{E}=E_1\cup E_2\cup (V_1\times V_2)$. 
\end{definition}

Or, put simply, it is the graph given by taking $G_1\otimes G_2$, then adding in an edge from each vertex in $G_1$ to each vertex in $G_2$. Equivalently, it can be defined as
\begin{displaymath}
G_1\times G_2=(G_1^C\otimes G_2^C)^C
\end{displaymath}

\begin{remark} The use of $\otimes$ and $\times$ to denote the operations of disjoint union and graph join respectively do not coincide with the notation used in graph theory. Graph union is often denoted as $G_1 \cup G_2$. Further, the graph join, sometimes called ``graph sum", is denoted $G_1+G_2$, (although the meaning of ``graph sum" can also vary depending on the literature). However, the notation $\lbrace \otimes,\times\rbrace$ was chosen in place of $\lbrace\cup,+\rbrace$ for consistency with the notation for Weil algebras.
\end{remark}

\begin{definition} An {\em independent set} $U$ of $G$ is a (possibly empty) subset of $V$ for which no two distinct vertices in $U$ have an edge between them (or equivalently, the full subgraph of $G$ induced by $U$ is discrete).
\end{definition}

\begin{definition} Conversely, a {\em clique} $U'$ of $G$ is a (possibly empty) subset  of $V$ for which any two distinct vertices in $U'$ have an edge between them (or equivalently, the full subgraph of $G^C$ induced by $U'$ is discrete).
\end{definition}

\begin{remark} Given a graph $G$, an independent set $U$ of $G$ is also a clique of $G^C$.
\end{remark}

We can actually use the notion of cliques and independent sets to form new graphs from existing ones.

\begin{definition}\label{igra}
Given a graph $G=(V,E)$, define $\ind(G)$ to be the graph given by:
\begin{quote}

\item[$\bullet$] Vertices: the independent sets of $G$;
\item[$\bullet$] Edges: given any two distinct independent sets $U_1$ and $U_2$ of $G$, there is an edge between them in $\ind(G)$ when $\exists ~x\in U_1,~y\in U_2$ such that either there is an edge between $x$ and $y$ in $G$ or $x=y$ (i.e. $U_1\cap U_2 \neq \phi$).
\end{quote}
\end{definition}

\begin{definition}\label{cgra}
Given a graph $G=(V,E)$, define $\cl(G)$ to be the graph given by:

\begin{quote}
\item[$\bullet$] Vertices: the cliques of $G$;
\item[$\bullet$] Edges: given any two distinct cliques $U_1$ and $U_2$ of $G$, there is an edge between them in $\cl(G)$ whenever their union $U_1\cup U_2$ is also a clique of $G$ (note that there is no requirement for $U_1$ and $U_2$ to be disjoint).
\end{quote}
\end{definition}

\begin{definition}
Given a graph $G=(V,E)$, define $\ind_+(G)$ to be the full subgraph of $\ind(G)$ where the vertices are the non-empty independent sets of $G$.
\end{definition}

These will become useful later in our description of morphisms of Weil algebras.

\subsection{Graphs and Weil algebras} \textcolor{white}{all work and no play makes jack a dull boy}

Each graph $G$ determines a Weil algebra (which we shall refer to as $k[G]$, which we will note is \textit{not} a group algebra in any way) with generators the vertices of $G$, with $uv=0$ whenever $(u,v)$ is an edge of $G$ and with $v^2=0$ for each vertex $v$.

Moreover, it is precisely Weil algebras of the form
\begin{displaymath}
C=k[c_1,\dots,c_n]/ \Big{\lbrace} c_ic_j \hspace{0.1cm} \Big{|} \hspace{0.1cm} c_i\backsim c_j \hspace{0.1cm} \Big{\rbrace},
\end{displaymath}
for a symmetric, reflexive relation $\backsim$, that arise in this way for a unique graph $G$. Conversely, given such a Weil algebra $C$, we can define $G_C$ to be the graph with vertex set $\{c_1,\dots,c_n\}$ and an edge $(c_i,c_j)$ whenever $c_ic_j=0$ for all $i\neq j$.

For example, we have

\begin{center}
\begin{tabular}{ccc}
  \hline
  \head{Weil algebra} & \head{Presentation} & \head{Graph}  \\
  \hline \\
  \vspace{0.2cm}
  $k$ & $k[~]$ &  \\
  \vspace{0.5cm}
  $W$ & $k[x]/x^2$ &
\begin{tikzpicture}[scale=0.75,transform shape]
\coordinate (X1) at (0,0) {};
\draw (X1) +(-7pt,-7pt) rectangle +(6.5pt,8.5pt);
\node at (X1) {1};
\end{tikzpicture}
\\
\vspace{0.5cm}
$2W$ & $k[x_1,x_2]/x_1^2,x_2^2$ &
\begin{tikzpicture}[scale=0.75,transform shape]
\coordinate (X1) at (0,0) {};
\draw (X1) +(-7pt,-7pt) rectangle +(6.5pt,8.5pt);
\node at (X1) {1};

\coordinate (X2) at (2,0) {};
\draw (X2) +(-7pt,-7pt) rectangle +(6.5pt,8.5pt);
\node at (X2) {2};
\end{tikzpicture}
\\
\vspace{0.5cm}
$W^2$ & $k[x_1,x_2]/x_1^2,x_2^2,x_1x_2$ &
\begin{tikzpicture}[scale=0.75,transform shape]
\coordinate (X1) at (0,0) {};
\draw (X1) +(-7pt,-7pt) rectangle +(6.5pt,8.5pt);
\node at (X1) {1};

\coordinate (X2) at (2,0) {};
\draw (X2) +(-7pt,-7pt) rectangle +(6.5pt,8.5pt);
\node at (X2) {2};

\draw[-,>=triangle 45] ($(X1)!\csize!(X2)$) -- ($(X2)!\csize!(X1)$);
\end{tikzpicture}
\\
\vspace{0.5cm}
$3W$ & $k[x_1,x_2,x_3]/x_1^2,x_2^2,x_3^2$ &
\begin{tikzpicture}[scale=0.75,transform shape]
\coordinate (X1) at (1,0.7) {};
\draw (X1) +(-7pt,-7pt) rectangle +(6.5pt,8.5pt);
\node at (X1) {1};

\coordinate (X2) at (0,-0.7) {};
\draw (X2) +(-7pt,-7pt) rectangle +(6.5pt,8.5pt);
\node at (X2) {2};

\coordinate (X3) at (2,-0.7) {};
\draw (X3) +(-7pt,-7pt) rectangle +(6.5pt,8.5pt);
\node at (X3) {3};
\end{tikzpicture}
\\
\vspace{0.5cm}
$W\times 2W$ & $k[x_1,x_2,x_3]/x_1^2,x_2^2,x_3^2,x_1x_2,x_1x_3$ &
\begin{tikzpicture}[scale=0.75,transform shape]
\coordinate (X1) at (1,0.7) {};
\draw (X1) +(-7pt,-7pt) rectangle +(6.5pt,8.5pt);
\node at (X1) {1};

\coordinate (X2) at (0,-0.7) {};
\draw (X2) +(-7pt,-7pt) rectangle +(6.5pt,8.5pt);
\node at (X2) {2};

\coordinate (X3) at (2,-0.7) {};
\draw (X3) +(-7pt,-7pt) rectangle +(6.5pt,8.5pt);
\node at (X3) {3};

\draw[-,>=triangle 45] ($(X1)!\csize!(X2)$) -- ($(X2)!\csize!(X1)$);
\draw[-,>=triangle 45] ($(X1)!\csize!(X3)$) -- ($(X3)!\csize!(X1)$);
\end{tikzpicture}
\\
\vspace{0.5cm}
$W^2\otimes W$ & $k[x_1,x_2,x_3]/x_1^2,x_2^2,x_3^2,x_1x_2$ &
\begin{tikzpicture}[scale=0.75,transform shape]
\coordinate (X1) at (1,0.7) {};
\draw (X1) +(-7pt,-7pt) rectangle +(6.5pt,8.5pt);
\node at (X1) {1};

\coordinate (X2) at (0,-0.7) {};
\draw (X2) +(-7pt,-7pt) rectangle +(6.5pt,8.5pt);
\node at (X2) {2};

\coordinate (X3) at (2,-0.7) {};
\draw (X3) +(-7pt,-7pt) rectangle +(6.5pt,8.5pt);
\node at (X3) {3};

\draw[-,>=triangle 45] ($(X1)!\csize!(X2)$) -- ($(X2)!\csize!(X1)$);
\end{tikzpicture}
\\
\vspace{0.5cm}
$W^3$ & $k[x_1,x_2,x_3]/x_1^2,x_2^2,x_3^2,x_1x_2,x_1x_3, x_2x_3$ &
\begin{tikzpicture}[scale=0.75,transform shape]
\coordinate (X1) at (1,0.7) {};
\draw (X1) +(-7pt,-7pt) rectangle +(6.5pt,8.5pt);
\node at (X1) {1};

\coordinate (X2) at (0,-0.7) {};
\draw (X2) +(-7pt,-7pt) rectangle +(6.5pt,8.5pt);
\node at (X2) {2};

\coordinate (X3) at (2,-0.7) {};
\draw (X3) +(-7pt,-7pt) rectangle +(6.5pt,8.5pt);
\node at (X3) {3};

\draw[-,>=triangle 45] ($(X1)!\csize!(X2)$) -- ($(X2)!\csize!(X1)$);
\draw[-,>=triangle 45] ($(X1)!\csize!(X3)$) -- ($(X3)!\csize!(X1)$);
\draw[-,>=triangle 45] ($(X2)!\csize!(X3)$) -- ($(X3)!\csize!(X2)$);
\end{tikzpicture}
\\
  \hline
\end{tabular}
\end{center}

Then for Weil algebras that can be described by such symmetric relations, the product of these Weil algebras is equivalent to taking the graph join and the coproduct is equivalent to the disjoint union of the corresponding graphs. 

More explicitly, for graphs $G$ and $H$, we have
\begin{align*}
k[G]\otimes k[H] &= k[G\otimes H]\\
k[G]\times k[H] &= k[G\times H]
\end{align*}

Then, to say we want only those Weil algebras that can be constructed from iterations of product and coproduct of $W$ is to say we want only those Weil algebras induced by graphs which can be constructed from iterations of graph join and disjoint union of the one-point graph (note that the empty graph is the identity for both graph join and disjoint union, and that the Weil algebra $k$ is the identity for product and coproduct).

Such graphs have been extensively studied and are known as \textit{cographs} (complement reducible graphs) and various characterisations have been given; see \cite{cogr}, for example.

So now we know exactly which Weil algebras we want to include in the subcategory \peil: all those that correspond to cographs.

\begin{remark} From here onwards, we shall only refer to those Weil algebras that correspond to cographs.
\end{remark}

\begin{remark}
For a (co)graph $G$ and the corresponding Weil algebra $k[G]$, a non-empty independent set of $G$ (i.e. a vertex of $\ind_+(G)$) is precisely a non-zero monomial term in the generators of $k[G]$.
\end{remark}

\section{The morphisms of \peil}

We said earlier that \peil was to be a full subcategory of \weil, and further mentioned that any morphism of \peil can be constructed from the ingredients of tangent structure. More specifically, we will also need composition, as well as the operations of $\times$ and $\otimes$ of these maps and the universal property of certain pullbacks (recall that products in \weil can be regarded as pullbacks of the augmentations).

It will turn out that we do not require the universal property of coproduct, but rather we shall consider \peil as a monoidal category with respect to $\otimes$ (and $k$ the unit for this operation).

Now recall that, ultimately, the tangent bundles will have the structure of $k$-modules for whatever the choice of $k$ may be; \br results in the bundles being real vector spaces, \bz yields abelian group bundles and \bn results in additive (commutative monoid) bundles. Although we are using the additive bundles of \cite{cocr} as motivation, it is more convenient to start by taking $k$ to be the rig $\{ 0,1\}$, where multiplication is the usual one and addition is given by $\maxx$, in particular $1+1=1$.

\subsection{Expressing maps using graphs}\label{maps as graphs} \textcolor{white}{all work and no play makes jack a dull boy}

Recall that to define a map between Weil algebras, it suffices to define how the map acts on each of the generators. So, consider two arbitrary Weil algebras
\begin{displaymath}
A=k[ a_1,\dots,a_n] /Q_A \textnormal{ and } B=k[ b_1,\dots,b_m] /Q_B
\end{displaymath}
of \peil (with corresponding (co)graphs $G_A$ and $G_B$) and an arbitrary map $f\colon A\ra B$.

Then for each generator $a_i$ of $A$, we can express $f(a_i)$ as a summation
\begin{displaymath}
f(a_i)=\sum\limits_{U\in \ind_+(G_B)} \alpha_U^i b_U~,
\end{displaymath}
summing over the vertices $U$ of $\ind_+(G_B)$ (recall these were the non-zero monomials in the generators of $B$), $\alpha_U^i\in k$ is a constant (taking value 0 or 1) and $b_U$ is the monomial corresponding to the independent set $U$.

Thus $f(a_i)$ is specified by choosing which of the $U$ have $\alpha_U^i=1$ (although there are some conditions that have to be satisfied, we will discuss these conditions later). We may then represent this information pictorially by circling these $U$'s on the graph $G_B$.

For example, consider the map $f\colon W\ra 3W$ given by $x\mapsto y_1y_2 + y_1y_3$. We can represent this in graph form:

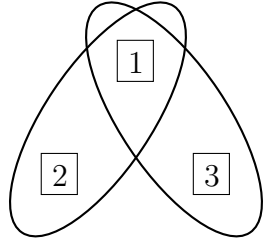
\begin{center}
\begin{tikzpicture}[xscale=1, yscale=1]
\newcommand{\elipseWidth}{0.7cm}

\coordinate (E1L) at (1.5cm,1.75cm) {};
\coordinate (E1R) at (-0.5cm,-1.25cm) {};

\coordinate (E2L) at (0.5cm,1.75cm) {};
\coordinate (E2R) at (2.5cm,-1.25cm) {};

\coordinate (E3L) at (-1cm,-0.5cm) {};
\coordinate (E3R) at (3cm,-0.5cm) {};

\pgfextractangle{\angle}{E1L}{E1R}
\draw[rotate around={\angle:($0.5*(E1L) + 0.5*(E1R)$)}, black, thick] let
\p1 = (E1L),
\p2 = (E1R),
\n1 = {veclen((\x2-\x1),(\y2-\y1))/2},
\n2 = {veclen((\elipseWidth),(0))}
in
($0.5*(E1L) + 0.5*(E1R)$) ellipse (\n1 and \n2);

\pgfextractangle{\angle}{E2L}{E2R}
\draw[rotate around={\angle:($0.5*(E2L) + 0.5*(E2R)$)}, black, thick] let
\p1 = (E2L),
\p2 = (E2R),
\n1 = {veclen((\x2-\x1),(\y2-\y1))/2},
\n2 = {veclen((\elipseWidth),(0))}
in
($0.5*(E2L) + 0.5*(E2R)$) ellipse (\n1 and \n2);

\coordinate (X1) at (1cm,1cm) {};
\draw (X1) +(-7pt,-7pt) rectangle +(6.5pt,8.5pt);
\node at (X1) {1};

\coordinate (X2) at (0cm,-0.5cm);
\draw (X2) +(-7pt,-7pt) rectangle +(6.5pt,8.5pt);
\node at (X2) {2};

\coordinate (X3) at (2cm,-0.5cm);
\draw (X3) +(-7pt,-7pt) rectangle +(6.5pt,8.5pt);
\node at (X3) {3};
\end{tikzpicture}
\end{center}
where each term of $f(x)$ is represented by circling the vertices that generate the term (so the term $y_1y_2$ is represented by the ellipse encompassing the vertices $1$ and $2$).

\begin{nota} For a graph $G$, let a \textit{circle} $U$ of $G$ simply mean an independent set of $G$, but regarded pictorially as some shape encompassing the relevant vertices.
\end{nota}

So, for a general map $f\colon A\ra B$, start by taking the generator $a_1$. Then take the graph $G_B$ for $B$, and for each $U$ with $\alpha_U^1=1$, we add onto $G_B$ a circle corresponding to $U$, and we do this for all $U$ with $\alpha_U^1=1$. Then repeat this process for each generator $a_i$, but (say) using different colours for different generators.

\begin{example} The map $f\colon 2W\ra 3W$ given by $x_1\mapsto y_1y_2+y_2y_3$ and $x_2\mapsto y_1+y_1y_3$ may be represented as
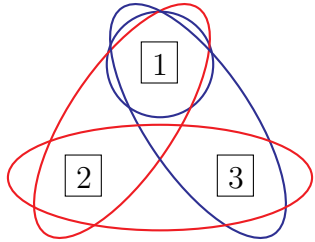
\begin{center}
\begin{tikzpicture}[xscale=1, yscale=1]
\newcommand{\elipseWidth}{0.7cm}

\coordinate (E1L) at (1.5cm,1.75cm) {};
\coordinate (E1R) at (-0.5cm,-1.25cm) {};

\coordinate (E2L) at (0.5cm,1.75cm) {};
\coordinate (E2R) at (2.5cm,-1.25cm) {};

\coordinate (E3L) at (-1cm,-0.5cm) {};
\coordinate (E3R) at (3cm,-0.5cm) {};

\coordinate (E4L) at (0.5cm,1cm) {};
\coordinate (E4R) at (1.5cm,1cm) {};

\draw[Blue, thick] (1cm,1cm) circle (0.7cm);

\pgfextractangle{\angle}{E1L}{E1R}
\draw[rotate around={\angle:($0.5*(E1L) + 0.5*(E1R)$)}, Red, thick] let
\p1 = (E1L),
\p2 = (E1R),
\n1 = {veclen((\x2-\x1),(\y2-\y1))/2},
\n2 = {veclen((\elipseWidth),(0))}
in
($0.5*(E1L) + 0.5*(E1R)$) ellipse (\n1 and \n2);

\pgfextractangle{\angle}{E2L}{E2R}
\draw[rotate around={\angle:($0.5*(E2L) + 0.5*(E2R)$)}, Blue, thick] let
\p1 = (E2L),
\p2 = (E2R),
\n1 = {veclen((\x2-\x1),(\y2-\y1))/2},
\n2 = {veclen((\elipseWidth),(0))}
in
($0.5*(E2L) + 0.5*(E2R)$) ellipse (\n1 and \n2);

\pgfextractangle{\angle}{E3L}{E3R}
\draw[rotate around={\angle:($0.5*(E2L) + 0.5*(E2R)$)}, Red, thick] let
\p1 = (E3L),
\p2 = (E3R),
\n1 = {veclen((\x2-\x1),(\y2-\y1))/2},
\n2 = {veclen((\elipseWidth),(0))}
in
($0.5*(E3L) + 0.5*(E3R)$) ellipse (\n1 and \n2);

\coordinate (X1) at (1cm,1cm) {};
\draw (X1) +(-7pt,-7pt) rectangle +(6.5pt,8.5pt);
\node at (X1) {\textnormal{1}};

\coordinate (X2) at (0cm,-0.5cm);
\draw (X2) +(-7pt,-7pt) rectangle +(6.5pt,8.5pt);
\node at (X2) {\textnormal{2}};

\coordinate (X3) at (2cm,-0.5cm);
\draw (X3) +(-7pt,-7pt) rectangle +(6.5pt,8.5pt);
\node at (X3) {\textnormal{3}};
\end{tikzpicture}
\end{center}
where $f(x_1)$ is represented in red and $f(x_2)$ is represented in blue.
\end{example}

\begin{nota} For a map $f\colon A\ra B$, let $\lbrace U\rbrace_f$ denote the graph $G_B$ together with a set $\big{\lbrace} (U,i)~|~\forall~\alpha_U^i=1\big{\rbrace}$, all of this regarded pictorially as a set of coloured circles on $G_B$.
\end{nota}

\begin{nota} Two circles $(U_1,i_1)$ and $(U_2,i_2)$ are \textit{distinct} if either $i_1\neq i_2$ (the circles are of different colours, but we may have $U_1=U_2$) or if $U_1\neq U_2$.
\end{nota}

\begin{remark} For a map $f\colon W\ra B$, then the circles $(U,i)$ of $\{ U\}_f$ will simply be referred to as $U$ (i.e. we omit the index $i$).
\end{remark}

So, to any map $f$ we can associate a graph with coloured circles. However, not all sets of circles of the graph $G_B$ are permissible. The conditions for $f$ being a map translate into conditions on $\lbrace U\rbrace_f$ as follows:
\begin{quote}
\item[$\bullet$] Given distinct circles $(U,i)$ and $(U',i)$ of $\lbrace U\rbrace_f$ (i.e. they correspond to the same generator $a_i$), then either $U\cap U'\neq \phi$ (so that they have a common vertex which becomes squared in the product $b_Ub_{U'}$) or we have $b_j\in U$ and $b_{j'}\in U'$ (with $j\neq {j'}$) such that $(b_jb_{j'})\in E_B$. This is because $a_i^2=0$, and so \begin{displaymath}
0=f(0)=f(a_i^2)=[f(a_i)]^2
\end{displaymath}

\item[$\bullet$] Further to the previous point, for a distinct pair of generators $a_{i_1}$ and $a_{i_2}$ of $A$ with $a_{i_1}a_{i_2}\in Q_A$ and two circles $(U,i_1)$ and $(U',i_2)$ of $\lbrace U\rbrace_f$, we have the same condition on $U$ and $U'$ as in the previous point. This is because 
\begin{displaymath}
0=f(0)=f(a_{i_1}a_{i_2})=f(a_{i_1})f(a_{i_2})
\end{displaymath}
\end{quote}

In either of the above cases, the condition placed on $U$ and $U'$ is precisely to say that $U$ and $U'$ have an edge between them in $\ind (G_B)$. This observation allows us to think of the maps in \peil in a rather convenient way using cliques and independent sets.

\subsection{Expressing maps using Cliques and Independent sets} \textcolor{white}{\hspace{0.1cm}}

We said in \textbf{\ref{maps as graphs}} that each summand $b_U$ of $f(a_i)$ was an independent set of $G_B$ (i.e. a vertex of $\ind_+(G)$), and moreover, we made the remark that for any two distinct summands $b_{U_1},b_{U_2}$ of $f(a_i)$, there is an edge between them in $\ind_+(G_B)$.

But this precisely means that the summands of $f(a_i)$ form a clique of $\ind_+(G_B)$, i.e. a vertex of $\cl(\ind_+(G_B))$ (and if $f(a_i)=0$, then we take the empty clique).We thus consider the graph
\begin{displaymath}
\cl(\ind_+(G_B))
\end{displaymath}
which we shall denote $\kappa (G_B)$.

With some minor additional calculations, it can be shown that to give a map $f\colon A\ra B$ of our Weil algebras is to give a map of graphs
\begin{displaymath}
G_A \ra \kappa (G_B).
\end{displaymath}

\begin{remark} In fact $\kappa$ defines an endofunctor on \gph, moreover this is canonically a monad. Further, if we take the Kleisli category induced by $\kappa$, then take the full subcategory with objects the cographs, then the bijection above gives an equivalence of categories between this full subcategory and \peil.
\end{remark}

\subsection{Construction of maps}\label{construction} \textcolor{white}{all work and no play makes jack a dull boy}

Recall that we had the maps $\epsilon_W,+,\eta,l$ and $c$ as our ingredients from tangent structure. These can also be represented in the form $\lbrace U\rbrace_f$ as in \textit{Figure 1} below, where the generators $x_1$ and $x_2$ are represented by red and blue respectively (where applicable).

\begin{center}
\begin{tabular}{ccc}
  \hline
  \head{Map} & \head{Action on Generators} & \head{Graph}  \\
  \hline \\
  \vspace{0.7cm}
$\epsilon\colon W\ra k$ & $x_1\mapsto 0$ & (k corresponds to the empty graph) \\

\vspace{1cm}

$id_W\colon W\ra W$ & $x\mapsto x$ &
\begin{tikzpicture}
\coordinate (X1) at (0cm,0cm) {};
\draw (X1) +(-7pt,-7pt) rectangle +(6.5pt,8.5pt);
\node at (X1) {1};

\draw[Red,thick] (0,0) circle (0.6cm);
\end{tikzpicture}
\\

\vspace{1cm}

$+\colon W^2\ra W$ & $x_1\mapsto x$, $x_2\mapsto x$ &
\begin{tikzpicture}
\coordinate (X1) at (0cm,0cm) {};
\draw (X1) +(-7pt,-7pt) rectangle +(6.5pt,8.5pt);
\node at (X1) {1};

\draw[Red,thick] (0,0) circle (0.6cm);
\draw[Blue,thick] (0,0) circle (0.7cm);
\end{tikzpicture}
\\

\vspace{1cm}

$\eta\colon k\ra W$ & (k has no generators) &
\begin{tikzpicture}
\coordinate (X1) at (0cm,0cm) {};
\draw (X1) +(-7pt,-7pt) rectangle +(6.5pt,8.5pt);
\node at (X1) {1};
\end{tikzpicture}
\\

\vspace{1cm}

$l\colon W\ra 2W$ & $x\mapsto x_1x_2$ &
\begin{tikzpicture}
\coordinate (X1) at (0cm,0cm) {};
\draw (X1) +(-7pt,-7pt) rectangle +(6.5pt,8.5pt);
\node at (X1) {1};

\coordinate (X2) at (2cm,0cm) {};
\draw (X2) +(-7pt,-7pt) rectangle +(6.5pt,8.5pt);
\node at (X2) {2};

\draw[Red,thick] (1cm,0cm) ellipse (2cm and 0.7cm);
\end{tikzpicture}
\\

\vspace{1cm}

$c\colon 2W\ra 2W$ & $x_1\mapsto x_2$, $x_2\mapsto x_1$ &
\begin{tikzpicture}
\coordinate (X1) at (0cm,0cm) {};
\draw (X1) +(-7pt,-7pt) rectangle +(6.5pt,8.5pt);
\node at (X1) {1};

\coordinate (X2) at (2cm,0cm) {};
\draw (X2) +(-7pt,-7pt) rectangle +(6.5pt,8.5pt);
\node at (X2) {2};

\draw[Red,thick] (X2) circle(0.6cm);
\draw[blue,thick] (X1) circle(0.6cm);
\end{tikzpicture}

\\
\hline
\vspace{0.1cm}
\end{tabular}

\textit{\small{Figure 1}}
\end{center}
\vspace{0cm}

Pictorially, given $\lbrace U\rbrace_f$ for some map $f\colon A\ra B$, we can naively interpret the composition on the left of the above operations as follows:
\begin{itemize}
\item[$\bullet$] $\epsilon$ corresponds to deleting a particular vertex in $G_B$ as well as any circles that go through that vertex.
\item[$\bullet$] $+$ corresponds to taking two vertices in $G_B$ joined by an edge and collapsing them to a single vertex. Circles that had contained either vertex (but not both) now contain the collapsed vertex instead.
\item[$\bullet$] $\eta$ corresponds to adding a new vertex to $G_B$, but has no effect on any of the existing circles.
\item[$\bullet$] $l$ corresponds to taking a single vertex of $G_B$ and splitting it into two vertices without an edge joining them, and any circle $U$ that contained the original vertex now contain both of the new vertices
\item[$\bullet$] $c$ corresponds to switching labels of unjoined vertices, and does nothing to the circles themselves.
\end{itemize}

These notions will become clearer in the following sections.

It will turn out that any map $f\colon A\ra B$ of \peil\ can be constructed in a canonical manner from the maps above. We shall break this process up into several steps.

\subsubsection{Step 1: Maps $W\ra nW$ with one circle}\label{one circle} \textcolor{red}{\hspace{0.2cm}}

We wish to consider a map of the form $f\colon W\ra nW$ with exactly one circle. Let us begin with an example.

\begin{example}
The map $f\colon W\rightarrow 5W$ given by $x\mapsto x_1x_3x_4$ may be represented as
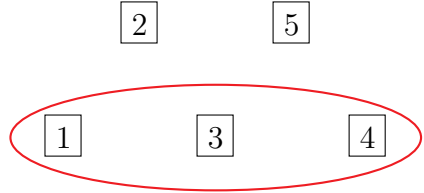
\begin{center}
\begin{tikzpicture}
\coordinate (X1) at (0cm,0cm) {};
\draw (X1) +(-7pt,-7pt) rectangle +(6.5pt,8.5pt);
\node at (X1) {\textnormal{1}};

\coordinate (X3) at (2cm,0cm) {};
\draw (X3) +(-7pt,-7pt) rectangle +(6.5pt,8.5pt);
\node at (X3) {\textnormal{3}};

\coordinate (X4) at (4cm,0cm) {};
\draw (X4) +(-7pt,-7pt) rectangle +(6.5pt,8.5pt);
\node at (X4) {\textnormal{4}};

\coordinate (X2) at (1cm,1.5cm) {};
\draw (X2) +(-7pt,-7pt) rectangle +(6.5pt,8.5pt);
\node at (X2) {\textnormal{2}};

\coordinate (X5) at (3cm,1.5cm) {};
\draw (X5) +(-7pt,-7pt) rectangle +(6.5pt,8.5pt);
\node at (X5) {\textnormal{5}};

\draw[Red,thick] (2cm,0cm) ellipse (2.7cm and 0.7cm);
\end{tikzpicture}
\end{center}

Define a map $\tilde{f}$ as the composite
\begin{equation*}
  \xymatrix@C+2em@R+2em{
W \ar[r]^-{l} & 2W \ar[r]^-{W\otimes l} & 3W
  }   
\end{equation*}
\begin{displaymath}
x\longmapsto x_1x_2\longmapsto x_1x_2x_3
\end{displaymath}

Then $\lbrace U\rbrace_{\tilde{f}}$ is
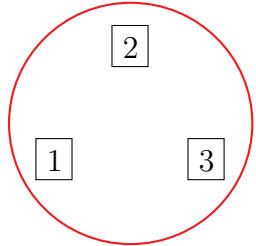
\begin{center}
\begin{tikzpicture}
\coordinate (X1) at (0cm,0cm) {};
\draw (X1) +(-7pt,-7pt) rectangle +(6.5pt,8.5pt);
\node at (X1) {\textnormal{1}};

\coordinate (X3) at (2cm,0cm) {};
\draw (X2) +(-7pt,-7pt) rectangle +(6.5pt,8.5pt);
\node at (X2) {\textnormal{2}};

\coordinate (X2) at (1cm,1.5cm) {};
\draw (X3) +(-7pt,-7pt) rectangle +(6.5pt,8.5pt);
\node at (X3) {\textnormal{3}};

\draw[Red,thick] (1cm,0.5cm) circle (1.6cm);
\end{tikzpicture}
\end{center}
i.e. the single circle includes all 3 vertices.

Now define a map $g$ as the map
\begin{align*}
W\otimes\eta\otimes W\otimes W\otimes\eta \colon &3W\ra 5W \\
x_1 & \mapsto y_1 \\
x_2 & \mapsto y_3 \\
x_3 & \mapsto y_4~.
\end{align*}

Then the composite $g\circ\tilde{f}$ is precisely the original map $f$.
\end{example}

This idea will extend to general maps of the form $f\colon W\ra nW$ with exactly one circle $U$. We first use a series of $l$'s to get the correct number of vertices in our circle. Then, we can post-compose with an appropriate combination of $\eta$'s and identities as above to obtain the required map $f$.

\subsubsection{Step 2: Arbitrary maps $W\ra nW$}\label{multiple circles} \textcolor{white}{\small{all work and no play makes jack a dull boy}}

We now discuss maps of the form $f\colon W\ra nW$ with an arbitrary (but finite) number of circles. Note that if there are no circles in $\{ U\}_f$ (i.e. $x\mapsto 0$), then the $f$ is given by (say) the composite
\begin{equation*}
  \xymatrix@C+2em@R+2em{
W \ar[r]^-{\epsilon} & k \ar[r]^-{\eta} & W \ar[r]^-{W\otimes \eta} & \dots \ar[r]^{(n-1)W\otimes \eta}  & nW
  }   
\end{equation*}
i.e. the zero map.

\begin{proposition} \label{prop5.3}
Any map $f\colon W\ra nW$ can be constructed from the generating maps.
\end{proposition}

\begin{proof}
We will do this by induction on $m$, the number of circles of $f$. We have already treated the cases $m=0$ and $m=1$, so suppose that $m>1$.

Explicitly, this means that $f(x)$ is a polynomial in the generators of $nW$ (which we will call $y_1,\dots,y_n$) with $m$ terms.

Recall that for $f$ to be a valid map, since the codomain is $nW$ (or equivalently, the corresponding graph is discrete), then any two distinct terms of $f(x)$ must have (at least) one generator $y_i$ in common. Let $t$ and $t'$ be distinct terms, and without loss of generality, suppose $y_n$ is a common generator.

Now define a map
\begin{displaymath}
f'\colon W\ra (n-1)W\otimes W^2
\end{displaymath}
where $W^2=k[y_n,\widetilde{y_n}]/y_n^2,\widetilde{y_n}^2,y_n\widetilde{y_n}$, with $f'(x)$ having the same expression as $f(x)$, except that the $y_n$ in term $t'$ is replaced with $\widetilde{y_n}$. It is a routine task to check that this is a valid map. Furthermore, the composite
\begin{equation*}
  \xymatrix@C+2em@R+2em{
W \ar[r]^-{f'} & (n-1)W\otimes W^2 \ar[r]^-{(n-1)W\otimes +} & nW
  }   
\end{equation*}
will return the original map $f$. As such, it suffices to show that $f'$ can be constructed from the ingredient maps.

But the codomain of $f'$, $(n-1)W\otimes W^2$, is the pullback
\begin{equation*}
  \xymatrix@C+2em@R+2em{
\ar @{} [dr] | \pbc (n-1)W\otimes W^2~ \ar[r]^-{(n-1)W\otimes\pi_1} \ar[d]_{(n-1)W\otimes\pi_2} & nW \ar[d]^{(n-1)W\otimes\epsilon_W} \\
nW \ar[r]_-{(n-1)W\otimes\epsilon_W} & (n-1)W
}
\end{equation*}
As such, to construct $f'$ it suffices to construct each of the composites
\begin{displaymath}
\big{(}(n-1)W\otimes\pi_i\big{)}\circ f'\colon W\ra nW; ~i\in \{1,2\} .
\end{displaymath}
But now each composite has strictly fewer terms (the first projection removes term $t'$ whilst the second projection removes term $t$, as well as any other terms containing $y_n$).

Now just repeat the process inductively.
\end{proof}

Consider the following example.

\begin{example} Consider the map
\begin{align*}
f\colon W &\ra 3W \\
x &\mapsto x_1x_2+x_1x_3+x_2x_3
\end{align*}
Then $\lbrace U\rbrace_f$ is
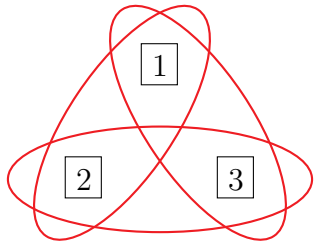
\begin{center}
\begin{tikzpicture}[xscale=1, yscale=1]
\newcommand{\elipseWidth}{0.7cm}

\coordinate (E1L) at (1.5cm,1.75cm) {};
\coordinate (E1R) at (-0.5cm,-1.25cm) {};

\coordinate (E2L) at (0.5cm,1.75cm) {};
\coordinate (E2R) at (2.5cm,-1.25cm) {};

\coordinate (E3L) at (-1cm,-0.5cm) {};
\coordinate (E3R) at (3cm,-0.5cm) {};

\pgfextractangle{\angle}{E1L}{E1R}
\draw[rotate around={\angle:($0.5*(E1L) + 0.5*(E1R)$)}, Red, thick] let
\p1 = (E1L),
\p2 = (E1R),
\n1 = {veclen((\x2-\x1),(\y2-\y1))/2},
\n2 = {veclen((\elipseWidth),(0))}
in
($0.5*(E1L) + 0.5*(E1R)$) ellipse (\n1 and \n2);

\pgfextractangle{\angle}{E2L}{E2R}
\draw[rotate around={\angle:($0.5*(E2L) + 0.5*(E2R)$)}, Red, thick] let
\p1 = (E2L),
\p2 = (E2R),
\n1 = {veclen((\x2-\x1),(\y2-\y1))/2},
\n2 = {veclen((\elipseWidth),(0))}
in
($0.5*(E2L) + 0.5*(E2R)$) ellipse (\n1 and \n2);

\pgfextractangle{\angle}{E3L}{E3R}
\draw[rotate around={\angle:($0.5*(E2L) + 0.5*(E2R)$)}, Red, thick] let
\p1 = (E3L),
\p2 = (E3R),
\n1 = {veclen((\x2-\x1),(\y2-\y1))/2},
\n2 = {veclen((\elipseWidth),(0))}
in
($0.5*(E3L) + 0.5*(E3R)$) ellipse (\n1 and \n2);

\coordinate (X1) at (1cm,1cm) {};
\draw (X1) +(-7pt,-7pt) rectangle +(6.5pt,8.5pt);
\node at (X1) {\textnormal{1}};

\coordinate (X2) at (0cm,-0.5cm);
\draw (X2) +(-7pt,-7pt) rectangle +(6.5pt,8.5pt);
\node at (X2) {\textnormal{2}};

\coordinate (X3) at (2cm,-0.5cm);
\draw (X3) +(-7pt,-7pt) rectangle +(6.5pt,8.5pt);
\node at (X3) {\textnormal{3}};
\end{tikzpicture}
\end{center}

We notice that we can construct this map if we can construct the map
\begin{align*}
f'\colon W &\ra (2W\otimes W^2)=k[x_1,x_2,x_3,x_4]/x_1^2,x_2^2,x_3^2,x_4^2,x_3x_4 \\
x &\mapsto x_1x_2+x_2x_3+x_1x_4
\end{align*}

where $\lbrace U\rbrace_{f'}$ is
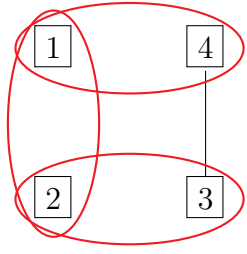
\begin{center}
\begin{tikzpicture}
\coordinate (X1) at (0cm,0cm) {};
\draw (X1) +(-7pt,-7pt) rectangle +(6.5pt,8.5pt);
\node at (X1) {\textnormal{1}};

\coordinate (X2) at (0cm,-2cm) {};
\draw (X2) +(-7pt,-7pt) rectangle +(6.5pt,8.5pt);
\node at (X2) {\textnormal{2}};

\coordinate (X3) at (2cm,-2cm) {};
\draw (X3) +(-7pt,-7pt) rectangle +(6.5pt,8.5pt);
\node at (X3) {\textnormal{3}};

\coordinate (X4) at (2cm,0cm) {};
\draw (X4) +(-7pt,-7pt) rectangle +(6.5pt,8.5pt);
\node at (X4) {\textnormal{4}};

\draw[-,>=triangle 45] ($(X3)!\csize!(X4)$) -- ($(X4)!\csize!(X3)$);

\draw[Red,thick] (1cm,0cm) ellipse (1.5cm and 0.6cm);
\draw[Red,thick] (1cm,-2cm) ellipse (1.5cm and 0.6cm);
\draw[Red,thick] (0cm,-1cm) ellipse (0.6cm and 1.5cm);
\end{tikzpicture}
\end{center}
since composition on the left with the map $2W\otimes +$ will return the map $f$ (recall that $+$ collapses two vertices joined by an edge down to a single vertex). 

But recall that for an arbitrary Weil algebra $A$, the functor $A\otimes\blank$ preserves pullbacks. So, we can regard $2W\otimes W^2$ as the functor $2W\otimes\underline{\hspace{0.3cm}}$ applied to the product $W^2$ regarded as a pullback over $k$, so that $2W\otimes W^2$ is part of a pullback diagram
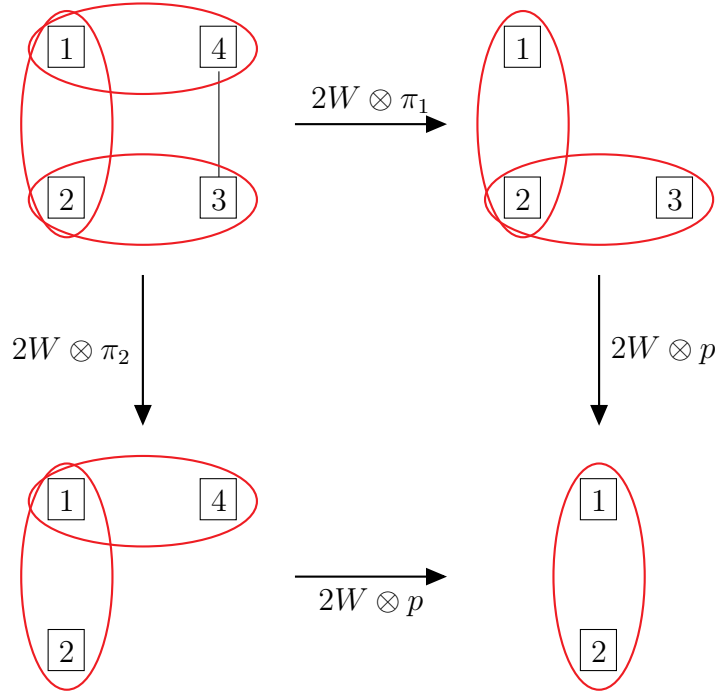
\begin{center} 
\begin{tikzpicture}
\coordinate (X1) at (0cm,0cm) {};
\draw (X1) +(-7pt,-7pt) rectangle +(6.5pt,8.5pt);
\node at (X1) {\textnormal{1}};

\coordinate (X2) at (0cm,-2cm) {};
\draw (X2) +(-7pt,-7pt) rectangle +(6.5pt,8.5pt);
\node at (X2) {\textnormal{2}};

\coordinate (X3) at (2cm,-2cm) {};
\draw (X3) +(-7pt,-7pt) rectangle +(6.5pt,8.5pt);
\node at (X3) {\textnormal{3}};

\coordinate (X4) at (2cm,0cm) {};
\draw (X4) +(-7pt,-7pt) rectangle +(6.5pt,8.5pt);
\node at (X4) {\textnormal{4}};

\draw[-,>=triangle 45] ($(X3)!\csize!(X4)$) -- ($(X4)!\csize!(X3)$);

\draw[Red,thick] (1cm,0cm) ellipse (1.5cm and 0.6cm);
\draw[Red,thick] (1cm,-2cm) ellipse (1.5cm and 0.6cm);
\draw[Red,thick] (0cm,-1cm) ellipse (0.6cm and 1.5cm);

\draw [->,>=triangle 45,thick] (3cm,-1cm) --node[above]{$2W\otimes\pi_1$} (5cm,-1cm);
\draw [->,>=triangle 45,thick] (1cm,-3cm) --node[left]{$2W\otimes\pi_2$} (1cm,-5cm);
\draw [->,>=triangle 45,thick] (3cm,-7cm) --node[below]{$2W\otimes p$} (5cm,-7cm);
\draw [->,>=triangle 45,thick] (7cm,-3cm) --node[right]{$2W\otimes p$} (7cm,-5cm);

\coordinate (X5) at (6cm,0cm) {};
\draw (X5) +(-7pt,-7pt) rectangle +(6.5pt,8.5pt);
\node at (X5) {\textnormal{1}};

\coordinate (X6) at (6cm,-2cm) {};
\draw (X6) +(-7pt,-7pt) rectangle +(6.5pt,8.5pt);
\node at (X6) {\textnormal{2}};

\coordinate (X7) at (8cm,-2cm) {};
\draw (X7) +(-7pt,-7pt) rectangle +(6.5pt,8.5pt);
\node at (X7) {\textnormal{3}};

\draw[Red,thick] (7cm,-2cm) ellipse (1.5cm and 0.6cm);
\draw[Red,thick] (6cm,-1cm) ellipse (0.6cm and 1.5cm);
\draw[Red,thick] (0cm,-7cm) ellipse (0.6cm and 1.5cm);
\draw[Red,thick] (7cm,-7cm) ellipse (0.6cm and 1.5cm);
\draw[Red,thick] (1cm,-6cm) ellipse (1.5cm and 0.6cm);

\draw (0cm,-6cm) +(-7pt,-7pt) rectangle +(6.5pt,8.5pt);
\node at (0cm,-6cm) {\textnormal{1}};

\draw (2cm,-6cm) +(-7pt,-7pt) rectangle +(6.5pt,8.5pt);
\node at (2cm,-6cm) {\textnormal{4}};

\draw (0cm,-8cm) +(-7pt,-7pt) rectangle +(6.5pt,8.5pt);
\node at (0cm,-8cm) {\textnormal{2}};

\draw (7cm,-6cm) +(-7pt,-7pt) rectangle +(6.5pt,8.5pt);
\node at (7cm,-6cm) {\textnormal{1}};

\draw (7cm,-8cm) +(-7pt,-7pt) rectangle +(6.5pt,8.5pt);
\node at (7cm,-8cm) {\textnormal{2}};
\end{tikzpicture}
\end{center}
where the circles in the pullback diagram depict $f$ composed with the projections. So, to construct $f$ is equivalent to constructing each of $(2W\otimes\pi_1)\circ f$ and $(2W\otimes\pi_2)\circ f$. 

Notice that, whilst we now have to construct 2 maps, each has strictly fewer circles than $f$. So, we can then consider $(2W\otimes\pi_1)\circ f$ and apply the same trick of splitting the vertex where the circles intersect and thus constructing $(2W\otimes\pi_1)\circ f$ is equivalent to constructing 2 maps, but each now with only one circle, which we can do using \ref{one circle}.
\end{example}

\begin{remark} Note that maps of this form cannot give rise to non-intersecting circles as it would violate
\begin{displaymath}
f(x)\cdot f(x)=f(x^2)=f(0)=0
\end{displaymath}
\end{remark}

We can actually take this one step further. Suppose we have an arbitrary map $f\colon W\ra nW$ with $\{ U_f\}$ given. For each $i\in \{1,\dots,n\}$, let $m_i$ be the number of circles containing vertex $i$ (or equivalently, the number of terms of $f(x)$ containing the generator $y_i$). Then, in a similar manner as before, we can define a map
\begin{displaymath}
f'\colon W\ra W^{m_1}\otimes \dots \otimes W^{m_n}
\end{displaymath}
in such a way that $(+_{m_1}\otimes \dots\otimes +_{m_n})\circ f'=f$. Note that if $m_i=0$ for some $i$, then $W^0$ is simply the nullary product $k$.

Now, iteratively using the fact that 
\begin{equation*}
  \xymatrix@C+2em@R+2em{
\ar @{} [dr] | \pbc A\otimes (B_1\times B_2) \ar[r]^-{A\otimes \pi_1} \ar[d]_{A\otimes \pi_2} & A\otimes B_1 \ar[d]^{A\otimes\epsilon_1} \\
A\otimes B_2 \ar[r]_-{A\otimes\epsilon_2} & A
}
\end{equation*}
is a pullback for all $A$, $B_1$ and $B_2\in$ \peil, then it is relatively easy to show that $W^{m_1}\otimes \dots \otimes W^{m_n}$ is a limit of an appropriate diagram of tensor powers of $W$'s.

As such, $f$ decomposes immediately into a set of maps $\{f_j\colon W\ra nW\}$, each of the type described in \ref{one circle} (or a zero map).

\subsubsection{Step 3: Projection maps $A\ra W$}\label{projections} \textcolor{white}{\small{all work and no play makes jack a dull boy}}

Given an arbitrary object $A=k[a_1,\dots,a_n]/Q_A$ in \peil, we wish now to consider maps of the form $f\colon A\ra W$ with $a_i\mapsto x$ for some fixed $i$ and $a_j\mapsto 0~\forall ~j\neq i$.

\begin{proposition} Let $G$ be a non-empty cograph, and let a vertex $u$ of $G$ be given. Then, there is a unique morphism of Weil algebras
\begin{displaymath}
f_{G,u}\colon k[G]\ra W,
\end{displaymath}
sending $u$ to $x$ and all other generators to 0, which can be constructed from the maps in Figure 1 as well as the product projections.
\end{proposition}

\begin{proof} Uniqueness is immediate since we have specified the action of $f_{G,u}$ on each generator. We now prove existence, done recursively as follows:
\begin{quote}
\item[1)] If $G=\{\bullet\}$, then $f_{G,u}$ is the identity.
\item[2)] Given graphs $H$ and $J$, if $G=H\otimes J$ with $u\in H$, then define $f_{G,u}$ as
\begin{equation*}
\xymatrix@C+2em@R+2em{
k[G]=k[H\otimes J]= k[H]\otimes k[J] \ar[r]^{\hspace{1.9cm}f_{H,u}\otimes\epsilon_{k[J]}} & W\otimes k\cong W
}   
\end{equation*}
\item[3)] If $G=H\times J$ with $u\in H$, then define $f_{G,u}$ as
\begin{equation*}
\xymatrix@C+2em@R+2em{
k[G]=k[H\times J]= k[H]\times k[J] \ar[r]^{\hspace{1.9cm}\pi_{k[H]}} & k[H] \ar[r]^{f_{H,u}} & W
}   
\end{equation*}

\end{quote}
\end{proof}

\subsubsection{Step 4: Maps $A\ra nW$ with no intersecting circles}\label{AnW} \textcolor{white}{\hspace{0.1cm}}

Suppose now we are interested in maps of the form
\begin{displaymath}
f\colon A\ra nW
\end{displaymath}
with no intersecting circles. If $f$ is the zero map, then its construction is trivial. If $\{ U\}_f$ has only one circle $(U,i)$ for some $a_i$ (i.e. every other generator maps to 0), then we can construct $f$ by first taking a projection as in \ref{projections} which preserves $a_i$, then compose with a map created using \ref{one circle} which picks out exactly the circle $U$.

However, if $\{ U\}_f$ contains multiple (distinct) circles, then for a distinct pair $(U,j)$ and $(U',j')$ we first make the following observations:
\begin{quote}
\item[$\bullet$] $U\cap U'=\phi$, since $f$ has no intersecting circles
\item[$\bullet$] $j\neq j'$, i.e. these circles must be associated with distinct generators of $A$. Equivalently, each generator can correspond to at most one circle.
\item[$\bullet$] $a_ja_{j'}\notin Q_A$, i.e. there is no edge between $a_j$ and $a_{j'}$ in the graph $G_A$
\end{quote}

So, let \ca be the full subcategory of \peil consisting of all objects $A$ with the property that any map $A\ra nW$ with no intersecting circles is constructible from the generating maps. Now, $W\in \ca$ by \ref{one circle} and clearly $k\in\ca$.

For arbitrary $A_1$ and $A_2\in\ca$, let an arbitrary map $f\colon A_1\times A_2\ra nW$ with no intersecting circles be given. Without loss of generality, let $a$ be a generator of $A_1$ for which $f(a)\neq 0$. Let $a'$ be an arbitrary generator of $A_2$. We know that in $A_1\times A_2$, we have $aa'=0$ by definition. Therefore $f(aa')=f(a)f(a')=0$. But since the codomain of $f$ is $nW$ and $f$ has no intersecting circles, then we must have $f(a')=0$. This is true for all generators of $A_2$. Therefore $f$ factors through the projection $\pi_1\colon A_1\times A_2\ra A_1$. Thus $f$ is constructible. Thus we have $A_1\times A_2\in\ca$.

Now suppose that an arbitrary map $f\colon A_1\otimes A_2\ra nW$ with no intersecting circles is given. Then, with some appropriate post-composition with $c$'s, we can write $f=f_1\otimes f_2$, for an appropriate pair $f_1\colon A_1\ra rW$ and $f_2\colon A_2\ra (n-r)W$ neither of which have intersecting circles.Thus $f$ is constructible. Thus we have $A_1\otimes A_2\in\ca$.

Now, since \ca is a full subcategory of \peil containing $k$, $W$ and is closed under $\times$ and $\otimes$, then \ca is just \peil itself. Thus any map $A\ra nW$ with no intersecting circles is constructible.

\subsubsection{Step 5: Arbitrary maps $A\ra nW$}\label{aAnW}\textcolor{white}{\small{all work and no play makes jack a dull boy}}

Now that we are able to construct maps of the form $f\colon A\ra nW$ with no intersecting circles (i.e. those of \ref{AnW}), we can combine this process with the ideas of \ref{multiple circles} to construct arbitrary maps of the form $f\colon A\ra nW$.

For an arbitrary map $f\colon A\ra nW$, we can use a similar idea to that discussed in \ref{multiple circles} to first construct an analogous map $f'\colon A\ra W^{m_1}\otimes \dots \otimes W^{m_n}$ as follows:
\begin{itemize}
\item[1)] For each generator $a_i$ of $A$, take the polynomial $f(a_i)$ in the generators $z_1,\dots ,z_n$ of $nW$
\item[2)] Let $m_j$ be the total number of terms across all the polynomials $f(a_1)$ containing $z_j$ for $j=1,\dots,n$
\item[3)] Define the map $f'\colon A\ra W^{m_1}\otimes \dots \otimes W^{m_n}$ by specifying each $f'(a_i)$ to be $f(a_i)$, but in such a way that each generator of $W^{m_1}\otimes \dots \otimes W^{m_n}$ is used exactly once (in a similar fashion to the proof for Proposition \ref{prop5.3})
\end{itemize}

\begin{example}
Consider the map $f\colon 2W\ra 3W$ given as
\begin{align*}
x_1 &\mapsto y_1y_2+y_1y_3 \\
x_2 &\mapsto y_2y_3~.
\end{align*}

Noting that each generator $y_i$ appear in exactly two monomials, then we have the map $f'\colon 2W\ra W^2\otimes W^2\otimes W^2$ given as
\begin{align*}
x_1 &\mapsto y_1y_2+y'_1y_3 \\
x_2 &\mapsto y'_2y'_3
\end{align*}
\end{example}

Given $f'$, it is then easy to see that the triangle
\begin{equation*}
\xymatrix{
& W^{m_1}\otimes \dots \otimes W^{m_n} \ar[d]^{+_{m_1}\otimes \dots\otimes +_{m_n}} \\
A \ar[r]_f \ar[ur]^-{f'} & nW
}
\end{equation*}
commutes. Further, for each projection $\pi =\pi_{i_1}\otimes \dots\otimes\pi_{i_n}$, $\pi\circ f'\colon A\ra nW$ has no intersecting circles, and it suffices to consider each such composite.

\subsubsection{Step 6: Arbitrary maps $A\ra B$}\label{mapsAB}\textcolor{white}{\small{all work and no play makes jack a dull boy}}

We now have everything we need to construct (or decompose, depending on your perspective) arbitrary maps $f\colon A\ra B$.

We need only note that if the corresponding graph $G_B$ has edges, then since it is a cograph, it can be expressed non-trivially as $(G_1\times G_2)\otimes H$ (with $H$ possibly being the empty graph, see \cite{cogr} for more details). Correspondingly, $B=(k[G_1]\times k[G_2])\otimes k[H]$ is the pullback
\begin{equation*}
  \xymatrix{
\ar @{} [dr] | \pbc B \ar[r]^-{\pi_1\otimes k[H]} \ar[d]_{\pi_2\otimes k[H]} & k[G_1]\otimes k[H] \ar[d]^{\epsilon_1\otimes k[H]} \\
k[G_2]\otimes k[H] \ar[r]_-{\epsilon_2\otimes k[H]} & k[H]
}
\end{equation*}
and so $f\colon A\ra B$ decomposes into the pair $(\pi_i\otimes k[H])\circ f$; $i=1,2$, and note that the graphs $G_i\otimes H$ for the codomains each have fewer vertices than $G_B$. As such, we repeat this process until the codomains are all of the form $nW$, then apply the idea from \ref{aAnW}.

\subsection{Additional coefficients} \textcolor{white}{\small{all work and no play makes jack a dull boy}}

So, \peil\ is the full subcategory of \weil, and we have described above a process to construct any map using a relatively simple set of ingredients. However, our constructions relied on $k$ being $\btwo$, and hence we limit the permissible maps by removing the possibility of coefficients that are neither 0 nor 1.

\begin{proposition}
The canonical functor
\begin{displaymath}
\psi\colon\bn\textnormal{-}\peil\ra \btwo\textnormal{-}\peil
\end{displaymath}
induced by the morphism $\psi\colon\bn\ra\btwo$, is bijective on objects and full.
\end{proposition}

\begin{proof}
Bijectivity on objects is immediate. For any morphism $f\colon A\ra B$ of $\btwo\textnormal{-}\peil$, there is a corresponding map $g\colon A\ra B$ in $\bn\textnormal{-}\peil$ given by the same action on generators as $f$. Clearly, we then have $\psi g=f$.
\end{proof}

Now, if we wish to construct some map $g\colon A\ra B$ in $\bn\textnormal{-}\peil$, the process is the same as for the corresponding map $\psi g\colon A\ra B$, except that each circle of $\psi g$ will need an appropriate coefficient.

\subsection{Coefficients in $\mathbb{N}$}\label{coefficients} \textcolor{white}{\small{all work and no play makes jack a dull boy}}

We now discuss the case where $k=\mathbb{N}$ for the tangent structures of \cite{cocr}. Consider the following diagram
\begin{equation*}
  \xymatrix@C+2em@R+2em{
   & W \ar[dl]_-{id_W} \ar[dr]^-{id_W} \ar@{-->}[d]^{g_2} & \\
   W & W^2 \ar[l]^-{\pi_1} \ar[r]_-{\pi_2} & W
  }   
\end{equation*}
where the uniquely induced map $g_2$ is defined as $g_2(x)=x_1+x_2$ (more explicitly, we have $g_2=\Delta$). Define the following composite
\begin{equation*}
  \xymatrix@C+1em@R+1em{
W \ar[r]^-{g_2} & W^2 \ar[r]^-{+} & W
  }   
\end{equation*}
to be $\hat{g}_2$. Note that we now have $\hat{g}_2(x)=2x$.

Now, consider the diagram
\begin{equation*}
  \xymatrix@C+2em@R+2em{
   & W \ar[dl]_-{id_W} \ar[dr]^-{\hat{g}_2} \ar@{-->}[d]^{g_3} & \\
   W & W^2 \ar[l]^-{\pi_1} \ar[r]_-{\pi_2} & W
  }   
\end{equation*}
where $g_3(x)=x_1+2x_2$. Then we define the composite $\hat{g}_3$ in a similar fashion as before so that $\hat{g}_3(x)=3x$. 

We repeat this process iteratively to obtain a set of maps
\begin{displaymath}
\Big{\lbrace } \hat{g}_r\colon W\ra W~|~\forall~r\in\mathbb{N} \Big{ \rbrace }
\end{displaymath}
where $\hat{g}_r(x)=rx$ (for completeness, take $\hat{g}_1$ to be the identity and $\hat{g}_0$ to be the composite $\eta\circ\epsilon \colon W\ra\mathbb{N}\ra W$). We give this detailed construction of these coefficient maps because we will need precisely this process to induce the corresponding maps in the final section).

Then this set of maps along with the previous techniques allow us to construct arbitrary maps with coefficients in $\mathbb{N}$, since we can use these maps $\hat{g}_r$ as an intermediate step between the projection described in \ref{projections} and the one-circle maps in \ref{one circle} to get our coefficients.

More generally though, whatever form $k$ may take (say $\mathbb{Z}$ or $\mathbb{R}$), our ``generating maps" would contain sufficient extra maps to allow us to generate the corresponding coefficients.

To reiterate, taking $k=\mathbb{N}$, the category \peil\ is the full subcategory of \augalg, and we have a methodical process for generating both the objects and morphisms.

\subsection{Choices}\textcolor{white}{\small{all work and no play makes jack a dull boy}} \label{choices}

In \ref{multiple circles} and \ref{aAnW}, there was an element of choice involved, namely given a map $f\colon A\ra nW$, the corresponding map $f'\colon A\ra W^{m_1}\otimes \dots \otimes W^{m_n}$ required a choice as to which circle would correspond to which projection. Ultimately, this choice is inconsequential as different choices are (up to isomorphism) equivalent.

However, for the purposes of what we wish to do, we will assume that for each $f\colon A\ra nW$, there is some pre-determined choice that has already been made regarding the corresponding map $f'$.

This then implicitly equips each map $g\colon A\ra B$ of \peil with a set of instructions for its decomposition (and hence reconstruction from the generating components).

\subsection{The map $\Omega$}\label{omega} \textcolor{white}{\small{all work and no play makes jack a dull boy}}

We will require a very particular map we shall call $\Omega$ in the next section, but we will introduce it here.

Consider the (trivially) commuting diagram
\begin{equation*}
  \xymatrix{
A \ar[r]^{f} \ar@/_12pt/[rr]_{h=g\circ f} & B \ar[r]^{g} & nW
}
\end{equation*}
for arbitrary $f\colon A\ra B$ and $g\colon B\ra nW$.

Since both $g$ and $h$ have codomain $nW$, then they decompose (in the sense of \ref{aAnW}) as
\begin{equation*}
\xymatrix{
& W^{\beta_1}\otimes\dots\otimes W^{\beta_n} \ar[d]^{+_\beta} & & W^{\alpha_1}\otimes\dots\otimes W^{\alpha_n} \ar[d]^{+_\alpha} \\
B \ar[r]_g \ar[ur]^-{g'} & nW & A \ar[r]_h \ar[ur]^-{h'} & nW
}
\end{equation*}

So, our goal now is to define a map
\begin{displaymath}
\Omega\colon W^{\alpha_1}\otimes\dots\otimes W^{\alpha_n} \ra W^{\beta_1}\otimes\dots\otimes W^{\beta_n}
\end{displaymath}
so that
\begin{equation*}
\xymatrix{
A \ar[r]^-{h'} \ar[d]_f & W^{\alpha_1}\otimes\dots\otimes W^{\alpha_n} \ar[d]^\Omega \\
B \ar[r]_-{g'} & W^{\beta_1}\otimes\dots\otimes W^{\beta_n}
}
\end{equation*}
commutes.

Since $W^{\beta_1}\otimes\dots\otimes W^{\beta_n}$ is a limit (as discussed in \ref{multiple circles}), then it suffices to define each map $\Omega_{(r_1,\dots,r_n)}$ as below
\begin{equation*}
\xymatrix{
W^{\alpha_1}\otimes\dots\otimes W^{\alpha_n} \ar[drrr]^{\Omega_{(r_1,\dots,r_n)}} \ar@{-->}[d]_{\Omega} &&& \\
W^{\beta_1}\otimes\dots\otimes W^{\beta_n} \ar[rrr]_-{(r_1,\dots,r_n)=\pi_1\otimes \dots\otimes \pi_n} & & & nW~.
}
\end{equation*}
But to give $\Omega_{(r_1,\dots,r_n)}$, it suffices to say where each generator of $W^{\alpha_1}\otimes\dots\otimes W^{\alpha_n}$ is sent. Let $y_1$ be a generator of $W^{\alpha_1}$ (without loss of generality, let $\alpha_1\geq 1$). We will also refer to the generators of $nW$ as $z_1,\dots,z_n$. Observe that $+_\beta(y_1)=z_1$.

Recall from \ref{choices} the construction of $h'\colon A\ra W^{\alpha_1}\otimes\dots\otimes W^{\alpha_n}$. There is a unique circle $(U_1,a)$ for some generator $a$ of $A$ with $y_1\in U_1$ (and correspondingly, a unique circle $(U_1,a)$ of $h$ as well with $z_1\in U_1$). Recall also that $h=g\circ f$. Let
\begin{displaymath}
h(a)=U_1 +U_2+\dots~,
\end{displaymath}
where each $U_i$ is a monomial in the generators $z_1,\dots,z_n$. Similarly, let
\begin{displaymath}
f(a)=V_1+V_2+\dots~,
\end{displaymath}
where each $V_i$ is a monomial in the generators $\{ b_j\}$ of $B$.

Then (ignoring coefficients), since $g$ preserves addition and multiplication, we can express $(g\circ f)(a)$ as
\begin{align*}
(g\circ f)(a) &= g(f(a)) \\
&= g(V_1+V_2+\dots) \\
&= g(V_1)+g(V_2)+\dots \\
&= [\prod_{b_j\in V_1}g(b_j)]+[\prod_{b_j\in V_2}g(b_j)]+\dots~.
\end{align*}

But this needs to be equal to $h(a)$. In particular, $U_1$ must be somewhere in the expression for $(g\circ f)(a)$. Without loss of generality, suppose $U_1$ is contained in the first term
\begin{displaymath}
\prod_{b_j\in V_1}g(b_j).
\end{displaymath}

Now, for each $b_j\in V_1$, we must be able to choose precisely one circle $Q_j$ in such a way that
\begin{displaymath}
\bigcup_{b_j\in V_1} Q_j =U_1
\end{displaymath}
with the $Q_j$'s pairwise distinct. This is because for each $b_j\in V_1$, $g(b_j)$ is a polynomial in the generators $z_1,\dots,z_n$. Then, if the product of these polynomials (which in turn is another polynomial) is to contain a particular monomial (namely $U_1$), then this monomial must have arisen as the product of one monomial from each of the factor polynomials.

Moreover, since $z_1\in U_1$, then we also have $z_1\in Q_j$ for a unique $j$. Take $j=1$ so that $Q_1$ is one of the terms of the polynomial $g(b_1)$.
\begin{align*}
&\Rightarrow Q_1 \textnormal{ is a circle of } g \textnormal{ corresponding to } b_1 \\
&\Rightarrow \textnormal{In } g'\colon B\ra W^{\beta_1}\otimes\dots\otimes W^{\beta_n} \textnormal{, } \exists ! \textnormal{ generator }v \textnormal{ of } W^{\beta_1} \textnormal{ corresponding to the circle } Q_1 \\
&\Rightarrow \textnormal{Define } \Omega_{(r_1,\dots,r_n)}(y_1)=\left\{
    \begin{array}{ll}
      z_1 &;~(r_1,\dots,r_n) \text{ preserves } v \text{ (in particular, }r_1 \text{ preserves } v) \\ 
      0 &; \text{ otherwise} \\
    \end{array}
  \right.
\\
\end{align*}
and repeat for all generators of $W^{\beta_1}\otimes\dots\otimes W^{\beta_n}$.

In particular, note that since $\Omega$ can only assign a generator from any $W^{\alpha_i}$ to a generator of the corresponding $W^{\beta_i}$, then we have $\Omega =\Omega_1\otimes\dots\otimes\Omega_n$, for appropriate maps $\Omega_i\colon W^{\alpha_i}\ra W^{\beta_i}$.

\section{Linking back to Tangent structure}

In the previous section, we defined a full subcategory \peil\ of \weil\ and showed how to construct any given map in this subcategory, with the restriction that $k=\mathbb{N}$.

We shall conclude by linking these ideas about Weil algebras to tangent structure in a much more explicit manner.

\subsection{Constructing a functor $F:\peil\ra \en $}\textcolor{white}{\hspace{0.1cm}}

\begin{theorem}\label{bigtheorem}
Suppose we have a given category \cm. Regard \textnormal{End}$(\cm)$ as a monoidal category with respect to composition and \peil\ as monoidal with respect to coproduct.

Then to give a tangent structure $\mathbb{T}$ to \cm is equivalent (up to isomorphism) to giving a strong monoidal functor $F\colon \peil\ra \textnormal{End}(\cm)$ which satisfies the following conditions:

1) Given a product $A=A_1\times A_2$ of \peil, regarded as a pullback of the augmentations, and an arbitrary Weil algebra $B$, then $F$ preserves the pullback
\begin{equation*}
  \xymatrix{
\ar @{} [dr] | \pbc B\otimes A \ar[r]^{B\otimes\pi_1} \ar[d]_{B\otimes\pi_2} & B\otimes A_1 \ar[d]^{B\otimes\epsilon_1} \\
B\otimes A_2 \ar[r]_{B\otimes\epsilon_2} & B
}
\end{equation*}
i.e. it preserves all ``foundational pullbacks" of \peil (as defined in \textnormal{\ref{weil facts 2}}).

2) The equaliser
\begin{equation*}
  \xymatrix@C+1em@R+1em{
W^2 \ar[r]^{v} & 2W \ar@<.5ex>[rr]^{W\otimes\epsilon_W} \ar@<-.5ex>[rr]_{\eta_W\circ (\epsilon_W\otimes\epsilon_W)} & & W
}
\end{equation*}
is preserved.
\end{theorem}

\begin{proof}
Given such a functor $F$, the corresponding tangent structure is given as
\begin{displaymath}
\mathbb{T}=(FW,F\epsilon_W,F\eta_W,F+,Fl,Fc)
\end{displaymath}
and it can be readily checked that this satisfies all the necessary conditions to be a tangent structure.

We now prove that this process is bijective in a suitable up-to-isomorphism sense. Suppose we have a tangent structure $\mathbb{T}$ and wish to define a functor $F\colon \peil\ra \textnormal{End}(\cm)$.

Since $F$ is required to send the unit Weil algebra $k$ to $1_\cm$ and further, since $k$ is a zero object in \peil, it is convenient to consider the category $1_\cm/\super(\cm)/1_\cm$ (which we shall call $\cd$ for convenience). Explicitly, it has:
\vspace{-0.2cm}
\begin{quote}
\item[$\bullet$]\textit{Objects:} Triples $(R,\eta_R,\epsilon_R)$ of the form
\begin{equation*}
  \xymatrix{
1_M \ar[r]^{\eta_R} & R \ar[r]^{\epsilon_R} & 1_M
}
\end{equation*}

(a diagram in End$(\cm)$) subject to $\epsilon_R \circ\eta_R=id$.

\item[$\bullet$] \textit{Morphisms:} A morphism $\phi\colon (R,\eta_R,\epsilon_R)\ra (S,\eta_S,\epsilon_S)$ is a morphism 
\begin{displaymath}
\phi\colon R\ra S
\end{displaymath}
of End$(\cm)$ such that the triangles in
\begin{equation*}
  \xymatrix{
1_M \ar[d]_{\eta_R} \ar[dr]^{\eta_S} & \\
R \ar[r]^{\phi} \ar[d]_{\epsilon_R} & S \ar[dl]^{\epsilon_S} \\
1_M
}
\end{equation*}
commute.

Composition and identities are defined in an obvious manner. Where obvious, we will refer to a triple $(R,\eta_R,\epsilon_R)$ simply as $R$.
\end{quote}

In particular, note the following:
\vspace{-0.2cm}
\begin{quote}
\item[$\bullet$] $1_M$ is a zero object
\item[$\bullet$] There is an obvious forgetful functor $I\colon \cd\ra \textnormal{End}(\cm)$
\item[$\bullet$] There is a canonical monoidal operation on $\cd$ induced by composition (which we shall also call $\otimes$ for convenience) with unit $1_M$
\item[$\bullet$] Given objects $R$ and $S$ of \cd, if the pullback of $\epsilon_R$ and $\epsilon_S$ exists in End$(\cm)$ then this pullback will be the product of $R$ and $S$ in \cd.
\end{quote}

Further, if the category \cm is equipped with a tangent structure $\mathbb{T}$, then the triple $(T,\eta,p)$ is an object of \cd, as are $T^n$ and $T^{(m)}$, and it can shown that the morphisms $+,l,c$ are also morphisms of \cd.

Thus, to define our desired functor $F\colon \peil\ra \textnormal{End}(\cm)$, it suffices to define a corresponding functor $F'\colon \peil\ra \cd$ and then post-compose with the forgetful functor $I$, and every such functor $\peil\ra\textnormal{End}(\cm)$ will be of this form, As such, we shall refer to $F$ and $F'$ interchangeably.

We begin by giving `assignations' for $F$ on objects and morphisms of \peil, then we shall later prove functoriality of this assignation.

Now, we have $k\mapsto 1_\cm$ (preservation of unit) and define $W\mapsto T$, and thus have $F(\epsilon_W)=p$.

Additionally, with a slight abuse of notation, we will also ask that $(\eta_W,+,l,c)$ in \peil are assigned to $(\eta,+,l,c)$ of $\mathbb{T}$. 

We also define $F$ on objects recursively: given $A$ and $B$ in \peil with $F(A)=R$ and $F(B)=S$, then $F(A\otimes B)=R\otimes S$ and $F(A\times B)=R\times S$. Thus $F$ by construction will preserve both $\otimes$ (as compositions) and $\times$ (products of \peil), since every object of \peil has a canonical decomposition.

We also assign the projection maps for pullbacks in \peil to the corresponding projections in \cd.

Moreover, note that we have (up to isomorphism)
\begin{align*}
nW&\mapsto T^n\\
W^m&\mapsto T^{(m)}
\end{align*}

With the assignation of the generating maps in \peil to those of $\bt$ as well as the assignation of the objects in mind, we naively would like to use the ideas detailed in \ref{construction} to recursively define the action of $F$ on all other maps of \peil. However, at this stage, this presents the potential problem of whether this is well defined. Explicitly, given a map $f\colon A\ra B$ of \peil, if $B\cong C\otimes (B_1\times B_2)$ (in a non-trivial way), then we have a diagram
\small{
\begin{equation*}
  \xymatrix{
A \ar@/^12pt/[rrd]^{f_1=(C\otimes\pi_1)\circ f} \ar@{-->}[rd]^f \ar@/_12pt/[rdd]_{f_2=(C\otimes\pi_2)\circ f} & & \\
& \ar @{} [dr] | \pbc B\cong C\otimes (B_1\times B_2) \ar[r]^-{C\otimes\pi_1} \ar[d]_{C\otimes\pi_2} & C\otimes B_1 \ar[d]^{C\otimes\epsilon_{B_1}} \\
& C\otimes B_2 \ar[r]_-{C\otimes\epsilon_{B_2}} & C
}
\end{equation*}
}
in \peil, and we we require that the corresponding diagram (with a slight abuse of notation)
\begin{equation}\label{pull}
  \xymatrix@C+1.5em@R+1.5em{
FA \ar[r]^{Ff_1} \ar[d]_{Ff_2} & FC\otimes FB_1 \ar[d]^{FC\otimes\epsilon_{FB_1}} \\
FC\otimes FB_2 \ar[r]_{FC\otimes\epsilon_{FB_2}} & FC
}
\end{equation}
commutes in \cd. We will need this if we are to use the universal property of the pullback
\begin{equation*}
  \xymatrix@C+1.5em@R+1.5em{
\ar @{} [dr] | \pbc FB \ar[r]^{\pi_1} \ar[d]_{\pi_2} & FC\otimes FB_1 \ar[d]^{FC\otimes\epsilon_{FB_1}} \\
FC\otimes FB_2 \ar[r]_{FC\otimes\epsilon_{FB_2}} & FC
}
\end{equation*}
to induce $Ff\colon FA\ra FB$.

So, we will begin by using \ref{construction} to define the action of $F$ on as many maps as possible (in the sense of \ref{choices}). We will now prove that $F$ is functorial at least for these well-defined assignations. Explicitly, we will show that for morphisms $f,g,h$ in \peil with $Ff,Fg,Fh$ well defined, if $h=g\circ f$ in \peil then $Fh=Fg\circ Ff$ in $\cd$.

More specifically, suppose we have
\begin{equation*}
  \xymatrix@C+2em@R+2em{
A \ar[r]^{f} \ar@/_12pt/[rr]_{h=g\circ f} & B \ar[r]^{g} & C
}
\end{equation*}
in \peil. We wish to show that 
\begin{equation*}
  \xymatrix@C+2em@R+2em{
FA \ar[r]^{Ff} \ar@/_12pt/[rr]_{Fh} & FB \ar[r]^{Fg} & FC
}
\end{equation*}
commutes in \cd.

\subsubsection{Reducing $C$ to $nW$}\label{decompC}\textcolor{white}{\small{all work and no play makes jack a dull boy}}

If the Weil algebra $C$ has any edges in its corresponding graph, then recall from \ref{mapsAB} that $C$ can be written as $C=C'\otimes (C_1\times C_2)$ and thus we have the pullback
\begin{equation*}
  \xymatrix@C+1em@R+1em{
\ar @{} [dr] | \pbc C'\otimes (C_1\times C_2) \ar[r]^-{C'\otimes\pi_1} \ar[d]_{C'\otimes\pi_2} & C'\otimes C_1 \ar[d]^{C'\otimes\epsilon_1} \\
C'\otimes C_2 \ar[r]_-{C'\otimes\epsilon_2} & C'
}
\end{equation*}
and since $F$ is to preserve such pullbacks, then $FC$ must also be a pullback of this form.

So in \cd we have
\begin{equation*}
  \xymatrix{
& & & FC_1 \\
FA \ar[r]^{Ff} \ar@/_12pt/[rr]_{Fh} & FB \ar[r]^{Fg} & FC \ar[ur]^{\pi_1} \ar[dr]_{\pi_2} \\
& & & FC_2
}
\end{equation*}

But recall that $Fg$ was induced by the maps $F(\pi_1\circ g)$ and $F(\pi_2\circ g)$ (and similarly for $Fh$). It then suffices to show that $F(\pi_i\circ h)=F(\pi_i\circ g)\circ Ff$. Repeating this argument, it suffices to assume $C=nW$ (correspondingly that $FC=T^n$).

\subsubsection{Reducing $g\colon B\ra nW$ to have no intersecting circles}\textcolor{white}{\hspace{1cm}}

Consider now the map $g\colon B\ra nW$. Recall that from \ref{multiple circles}, $g$ factorises as the composite
\begin{equation*}
  \xymatrix{
& W^{\beta_1}\otimes\dots \otimes W^{\beta_n} \ar[d]^{+_{\beta}} \\
B \ar[r]_{g} \ar[ur]^{g'} & nW
}
\end{equation*}

Thus, $Fg$ is constructed as the corresponding composite
\begin{equation*}
  \xymatrix{
& T^{(\beta_1)}\otimes\dots \otimes T^{(\beta_n)} \ar[d]^{+_\beta} \\
FB \ar[r]_{Fg} \ar[ur]^{Fg'} & T^n
}
\end{equation*}
noting that since $+\colon T^{(2)}\ra T$ is an associative, commutative and unital map, then there is a well defined map $+_{\beta_i}\colon T^{(\beta_i)}\ra T$ for each $i$, and finally $+_\beta =+_{\beta_1}\otimes\dots\otimes +_{\beta_n}$. Similarly, the map $h$ is constructed as an appropriate composite
\begin{equation*}
  \xymatrix{
& T^{(\alpha_1)}\otimes\dots \otimes T^{(\alpha_n)} \ar[d]^{+_\alpha} \\
FA \ar[r]_{Fh} \ar[ur]^{Fh'} & T^n
}
\end{equation*}

We now have the following diagram
\begin{equation}\label{complicated diagram}
\xymatrix{
&& FB \ar[d]^{Fg'} \ar@/^16pt/[ddrr]^-{Fg} && \\
&& T^{(\beta_1)}\otimes\dots \otimes T^{(\beta_n)} \ar[drr]^-{+_\beta} && \\
FA \ar[rr]^-{Fh'} \ar[urr]^-{F(g'\circ f)} \ar@/^16pt/[uurr]^-{Ff} \ar@/_14pt/[rrrr]_-{Fh} && T^{(\alpha_1)}\otimes\dots \otimes T^{(\alpha_n)} \ar[rr]^-{+_\alpha} && T^n
}
\end{equation}
for which we wish to show the commutativity of the exterior. We already know the bottom triangle as well as top right triangle commute by construction. We begin with the innermost square
\begin{equation}\label{important square}
\xymatrix{
FA \ar[r]^-{F(g'\circ f)} \ar[d]_{Fh'} & T^{(\beta_1)}\otimes\dots \otimes T^{(\beta_n)} \ar[d]^{+_\beta} \\
T^{(\alpha_1)}\otimes\dots \otimes T^{(\alpha_n)} \ar@{-->}[ur]_\Omega \ar[r]_-{+_\alpha} & FB
}
\end{equation}
and and introduce the map $\Omega$ as defined in \ref{omega}. Recall that
\begin{displaymath}
\Omega\colon W^{\alpha_1}\otimes\dots\otimes W^{\alpha_n}\ra W^{\beta_1}\otimes\dots\otimes W^{\beta_n}
\end{displaymath}
assigned each generator of the domain to a particular generator in the codomain, and moreover we have $\Omega=\Omega_1\otimes\dots\otimes\Omega_n$ for $\Omega_i\colon W^{\alpha_i}\ra W^{\beta_i}$.

First, it now becomes rather routine to show that $+_\alpha=+_\beta\circ\Omega$ (i.e. the lower triangle in (\ref{important square})). To show $F(g'\circ f)=\Omega\circ Fh'$ in \cd, note first that $g'\circ f=\Omega\circ h'$ in \peil by construction. So equivalently, we can show that
\begin{equation*}
\xymatrix{
FA \ar[r]^-{Fh'} \ar@/_14pt/[rr]_{F(\Omega\circ h')} & T^{(\alpha_1)}\otimes\dots \otimes T^{(\alpha_n)} \ar[r]^\Omega & T^{(\beta_1)}\otimes\dots \otimes T^{(\beta_n)}
}
\end{equation*}
commutes.

But recall from \ref{omega} that $T^{(\beta_1)}\otimes\dots \otimes T^{(\beta_n)}$ is a limit (constructed as iterations of foundational pullbacks, hence preserved by $F$) with projections $(r_1,\dots ,r_n)$. As such, it suffices to show the commutativity of
\begin{equation*}
\xymatrix{
FA \ar[r]^-{Fh'} \ar@/_14pt/[rrr]_{F\big{(}(r_1,\dots,r_n)\circ\Omega\circ h'\big{)}} & T^{(\alpha_1)}\otimes\dots \otimes T^{(\alpha_n)} \ar[rr]^-{(r_1,\dots,r_n)\circ\Omega} & & T^n
}
\end{equation*}
for each $(r_1,\dots,r_n)$. But noting that $\Omega=\Omega_1\otimes\dots\otimes\Omega_n$ and $(r_1,\dots,r_n)=\pi_{r_1}\otimes\dots\otimes\pi_{r_n}$, and noting the form of each $\pi_{r_i}\circ\Omega_i\colon T^{(\alpha_1)}\otimes\dots \otimes T^{(\alpha_n)} \ra T$, then the commutativity of the upper triangle in (\ref{important square}) above is immediate.

Hence, all that remains is to show the commutativity of the upper left triangle of (\ref{complicated diagram}), namely the commutativity of
\begin{equation*}
\xymatrix{
& & FB \ar[d]^{Fg'} \\
FA \ar[urr]^{Ff} \ar[rr]_-{F(g'\circ f)} & & T^{(\beta_1)}\otimes\dots \otimes T^{(\beta_n)}
}
\end{equation*}

Again, since $T^{(\beta_1)}\otimes\dots \otimes T^{(\beta_n)}$ is a limit, it suffices to show the commutativity of
\begin{equation*}
\xymatrix{
& & & FB \ar[d]^{(r_1,\dots,r_n)\circ Fg'=F\big{(}(r_1,\dots,r_n)\circ g'\big{)}} \\
FA \ar[urrr]^{Ff} \ar[rrr]_-{F\big{(}(r_1,\dots,r_n)\circ g'\circ f\big{)}} & & & T^n
}
\end{equation*}
for each projection $(r_1,\dots,r_n)$. Finally, note that by construction, each map
\begin{displaymath}
(r_1,\dots,r_n)\circ g'\colon B\ra nW
\end{displaymath}
has no intersecting circles. Thus, it suffices to show the commutativity of 
\begin{equation*}
  \xymatrix{
FA \ar@/_12pt/[rr]_{Fh} \ar[r]^{Ff} & FB \ar[r]^{Fg} & T^n
}
\end{equation*}
for the case where $g$ has no intersecting circles.

\subsubsection{Reducing $B$ to $mW$} \textcolor{white}{\small{all work and no play makes jack a dull boy}}

Recall from \ref{mapsAB} that if the graph $G_B$ contains edges, then it can be expressed in the form $B'\otimes (B_1\times B_2)$, and is thus part of a pullback diagram
\begin{equation*}
  \xymatrix{
\ar @{} [dr] | \pbc B=B'\otimes (B_1\times B_2) \ar[r]^-{B'\otimes\pi_1} \ar[d]_{B'\otimes\pi_2} & B'\otimes B_1 \ar[d]^{B'\otimes\epsilon_{B_1}} \\
B'\otimes B_2 \ar[r]_-{B'\otimes \epsilon_{B_2}}  & B'
}
.
\end{equation*}

Recall also that, from \ref{AnW}, $g:B\ra nW$ having no intersecting circles implies, if $B$ has any edges, that $g$ must factorise through one of the pullback projections (and $Fg$ by definition also factorises through the corresponding projection).

Let $g$ factorise through $B'\otimes \pi_1$. Then we have
\begin{equation*}
  \xymatrix{
& & FB'\otimes FB_1 \ar[d]^{Fg'} \\
FA \ar[r]^{Ff} \ar@/_12pt/[rr]_{Fh} & FB \ar[r]^{Fg} \ar[ur]^{\pi_1} & T^n
},
\end{equation*}
and noting that since $B$ is a pullback of the type that $F$ preserves, then $FB$ is also part of a corresponding pullback.

Now, since $\pi_1\circ Ff=F(\pi_1\circ f)$ by definition, and since $g'$ has no intersecting circles ($\{U\}_{g'}$ has the same circles as $\{U\}_g$, factoring through one of the projections only removes ``unused" vertices of $B$), then we can repeat this argument until $B$ no longer has any edges.

It thus suffices to consider the commutativity of
\begin{equation*}
  \xymatrix@C+2em@R+2em{
FA \ar[r]^{Ff} \ar@/_12pt/[rr]_{Fh} & T^m \ar[r]^{Fg} & T^n
}
\end{equation*}
where $g$ has no intersecting circles.

\subsubsection{Removing the intersecting circles of $f$ and $h$}\textcolor{white}{\hspace{1cm}}

Let $\{y_1,\dots,y_m\}$ denote the generators of $mW$ and $\{z_1,\dots,z_n\}$ denote the generators of $nW$. So, in \peil\ we have $g\colon mW\ra nW$, and $\{U\}_g$ has no intersecting circles. Then, modulo some $c$'s to relabel generators and omitting the coefficient maps described in \ref{coefficients}, we have $g=g_1\otimes \dots\otimes g_m\otimes\eta_{sW}$, where each $g_i\colon W\ra s_iW$ is is uniquely determined by the value of $s_i$;
\begin{align*}
s_i=0~&\Rightarrow~g_i=\epsilon_W\colon W\ra k \\
s_i=1~&\Rightarrow~g_i=id_W\colon W\ra W \\
s_i=2~&\Rightarrow~g_i=l\colon W\ra 2W \\
s_i=3~&\Rightarrow~g_i=(W\otimes l)\circ l\colon W\ra 3W
\end{align*}
and so forth, with the $s$ and $s_i$ subject to
\begin{displaymath}
s+\sum\limits^{m}_{i=1}s_i=n
\end{displaymath}

Without loss of generality, we can let $s=0$ and $s_i>0~\forall~i$. This means that in \peil , no generator $y_i$ of $mW$ is sent to zero and every generator $z_j$ of $nW$ belongs to exactly one of the $m$ circles of $\{U\}_g$. This then defines a surjective function
\begin{displaymath}
\psi\colon \{z_1,\dots,z_n\}\ra\{ y_1,\dots,y_m\}.
\end{displaymath}
Without loss of generality, suppose that $\psi(z_1)=y_1$.

Suppose the maps $f$ and $h$ factorise as the composites
\begin{equation*}
\xymatrix{
& W^{\alpha_1}\otimes\dots \otimes W^{\alpha_n} \ar[d]^{+_\alpha} & & W^{\gamma_1}\otimes\dots \otimes W^{\gamma_nm} \ar[d]^{+_\gamma} \\
A \ar[r]_h \ar[ur]^-{h'} & nW & A \ar[r]_f \ar[ur]^-{f'} & mW
}
\end{equation*}
Firstly, this means that for the map $h$, there are precisely $\alpha_1$ circles (say $U_1,\dots,U_{\alpha_1}$) containing the generator $z_1$. But given what we've established about $g$, and noting that $h=g\circ f$, then the $z_1$ term in each of these $U_i$ must arise as a result of the generator $y_1$ (since $\psi(z_1)=y_1$). More explicitly, to each circle $U_i$ of $h$ containing $z_1$ we can associate a unique circle of $f$ containing $y_1$.

Conversely, for each circle $V_j$ of $f$ containing $y_1$, we have $g(V_j)\neq 0$ (moreover, $g(V_j)$ is a single circle) and $z_1\in g(V_j)$. Therefore the number of circles of $f$ containing $y_1$ (namely $\gamma_1$) is the same as the number of circles of $h$ containing $z_1$ (namely $\alpha_1$). Thus, if $\phi(z_i)=y_j$, then $\alpha_i=\gamma_j$.

We then define a map
\begin{displaymath}
\Gamma\colon W^{\gamma_1}\otimes\dots \otimes W^{\gamma_m}\ra W^{\alpha_1}\otimes\dots \otimes W^{\alpha_n}
\end{displaymath}
given as
\begin{equation}\label{Gamma}
\xymatrix{
& mW \ar[rr]^g & & nW \\
W^{\gamma_1}\otimes\dots \otimes W^{\gamma_m} \ar[ur]^{t} \ar@{-->}[rr]_-{\exists !\Gamma} & & W^{\alpha_1}\otimes\dots \otimes W^{\alpha_n} \ar[ur]_-{(r_1,\dots,r_n)}~, &
}
\end{equation}
where for fixed $(r_1,\dots,r_n)$, $(t_1,\dots,t_m)$ is determined as follows:
\begin{itemize}
\item[$\bullet$] Consider $(r_1,\dots,r_n)\circ h'\colon A\ra nW$. If this is the zero map, then $t$ is also the zero map.
\item[$\bullet$] If not, this means that there is at least one circle $U$ of $h$ (and hence $h'$) with each of its generators preserved by $(r_1,\dots,r_n)$. Moreover, if there are multiple circles, then they must be disjoint and each corresponds to a different generator of $A$ (see \ref{aAnW}). Without loss of generality, assume there is only one such circle $U$. Regard $U$ as a subset of $\{z_1,\dots,z_n\}$. Then we know $\psi(U)$ (the image of $U$ under $\psi$) is the unique circle of $f$ corresponding to $U$. Choose $t$ in such a way (moreover, in the unique way) that it preserves this circle $\psi(U)$ of $f$, but sends any $y_j\notin \psi(U)$ to 0.
\end{itemize}

Now, it is fairly routine (albeit tedious) to show that $F\Gamma$ (which we shall also call $\Gamma$ for simplicity) is well defined in \cd, and further, that
\begin{equation*}
\xymatrix{
W^{\gamma_1}\otimes\dots \otimes W^{\gamma_m} \ar[rr]^-{\Gamma} \ar[d]_{+_\gamma} & & W^{\alpha_1}\otimes\dots \otimes W^{\alpha_n} \ar[d]^{+_\alpha} \\
mW \ar[rr]_g & & nW
}
\end{equation*}
commutes in \peil (and that the corresponding diagram commutes in \cd). We now have the following diagram
\begin{equation*}
\xymatrix{
& & T^m \ar@/^14pt/[rrdd]^{Fg} & & \\
& & T^{(\gamma_1)}\otimes\dots\otimes T^{(\gamma_m)} \ar[u]_{+_\gamma} \ar[d]^\Gamma & & \\
FA \ar[rr]^-{Fh'} \ar@/^14pt/[uurr]^{Ff} \ar[urr]^-{Ff'} \ar@/_14pt/[rrrr]_{Fh} & & T^{(\alpha_1)}\otimes\dots\otimes T^{(\alpha_n)} \ar[rr]^-{+_\alpha} & & T^n~,
}
\end{equation*}
and note that to show the exterior commutes, it suffices to show that
\begin{equation*}
\xymatrix{
& T^{(\gamma_1)}\otimes\dots\otimes T^{(\gamma_m)} \ar[d]^\Gamma \\
FA \ar[ur]^{Ff'} \ar[r]_-{Fh'} & T^{(\alpha_1)}\otimes\dots\otimes T^{(\alpha_n)}
}
\end{equation*}
commutes. As $T^{(\alpha_1)}\otimes\dots\otimes T^{(\alpha_n)}$ is a limit with projections $(r_1,\dots,r_n)$, and since $\Gamma$ was given using (\ref{Gamma}), then it suffices to show that
\begin{displaymath}
(r_1,\dots,r_n)\circ\Gamma\circ Ff'=(r_1,\dots,r_n)\circ Fh'
\end{displaymath}
for all $(r_1,\dots,r_n)$. We then have the diagram
\begin{equation*}
\xymatrix@C+1em@R+1em{
FA \ar@/_20pt/[ddrrrr]_-{F\big{(}(r_1,\dots,r_n)\circ h'\big{)}} \ar[drr]^-{F(t\circ f')} \ar[dd]_{Fh'} \ar[rrrr]^-{Ff'} & & & & T^{(\gamma_1)}\otimes\dots\otimes T^{(\gamma_m)} \ar[d]^\Gamma \ar[dll]_{Ft} \\
& & T^m \ar[drr]^{Fg} & & T^{(\alpha_1)}\otimes\dots\otimes T^{(\alpha_n)} \ar[d]^{(r_1,\dots,r_n)} \\
T^{(\alpha_1)}\otimes\dots\otimes T^{(\alpha_n)} \ar[rrrr]_-{(r_1,\dots,r_n)} & & & & T^n
}
\end{equation*}
and to show the exterior commutes, we first note that the lower left triangle and right square commute by construction, and further that it is routine to check that the top triangle commutes. So all that remains is to verify the commutativity of the innermost triangle. But note that $t\circ f'$ by construction has no intersecting circles.

It thus suffices to consider the commutativity of
\begin{equation*}
  \xymatrix@C+2em@R+2em{
FA \ar[r]^{Ff} \ar@/_12pt/[rr]_{Fh} & T^m \ar[r]^{Fg} & T^n
}
\end{equation*}
where $g$ has no intersecting circles and sends no generator $y_i$ to zero, and $f$ also has no intersecting circles in \peil\ (nor does $h$).

\subsubsection{Decomposition of the Weil algebra $A$}\label{decompA}\textcolor{white}{\small{all work and no play makes jack a dull boy}}

We again inductively use the argument that as $f$ (and $h$) have no intersecting circles, it must factor through one of the projections of $A$, and thus it will suffice to take $A=T^\nu$, and thus $f$ has the form (again modulo some $c$'s and omitting coefficient maps of \ref{coefficients}) $f=f_1\otimes \dots\otimes f_\nu\otimes \eta_{rW}$ where each $f_i\colon W\ra r_iW$ is uniquely determined by the value of $r_i$ in an identical manner to that of $g=g_1\otimes \dots\otimes g_m\otimes \eta_{sW}$ from before.

It thus suffices to consider the commutativity of
\begin{equation*}
  \xymatrix@C+2em@R+2em{
T^\nu \ar[r]^{Ff} \ar@/_12pt/[rr]_{Fh} & T^m \ar[r]^{Fg} & T^n
}
\end{equation*}
where $g$ sends no generator $y_i$ of $mW$ to zero, and neither $f$ nor $g$ have any intersecting circles.

Now, it can be shown that the diagram
\begin{equation*}
  \xymatrix@C+2em@R+2em{
T \ar[r]^{F\hat{g}_c} \ar[d]_{F\hat{g}_{cd}} & T \ar[r]^l & T^2 \ar[d]^{F\hat{g}_d\otimes id} \\
T \ar[rr]_l & & T^2
}
\end{equation*}
commutes, where $\hat{g}_c\colon W\ra W$ and $\hat{g}_d\colon W\ra W$ are the coefficient maps detailed in \ref{coefficients}. Coupled with the commutativity of the diagram
\begin{equation*}
  \xymatrix@C+2em@R+2em{
T \ar[r]^l \ar[d]_l & T^2 \ar[d]^{l\otimes id} \\
T^2 \ar[r]_{id\otimes l} & T^3
}
\end{equation*}
then commutativity is now trivial.

So, we have shown that $Fh=Fg\circ Ff$ whenever $h=g\circ f$ and $F$ is well defined on each of $f,g,h$. It remains to show that $F$ is well defined on all maps of \peil.

\subsubsection{The Problem with Pullbacks} \textcolor{white}{\small{all work and no play makes jack a dull boy}}

As mentioned earlier, there may exist maps $f\colon A\ra B$ for which $Ff$ is not well defined, i.e. for $f_1,f_2$ below
\begin{equation*}
  \xymatrix{
A \ar@{-->}[dr]^f \ar@/_15pt/[rdd]_{f_2=\pi_2\circ f} \ar@/^15pt/[rrd]^{f_1=\pi_1\circ f} & & \\
& \ar @{} [dr] | \pbc B=(B_1\times B_2)\otimes C \ar[r]_-{\pi_1} \ar[d]^{\pi_2} & B_1\otimes C \ar[d]^{\epsilon\otimes C} \\
& B_2\otimes C \ar[r]_{\epsilon_2\otimes C} & C
}
\end{equation*}
the square
\begin{equation*}
\xymatrix{
FA \ar[r]^{Ff_1} \ar[d]_{Ff_2} & FB_1\otimes FC \ar[d]^{\epsilon_1\otimes FC} \\
FB_2\otimes C \ar[r]_{\epsilon_2\otimes FC} & FC
}
\end{equation*}
does not commute. We will now show that this cannot happen.

\begin{proposition}
$F$ is well defined for all maps $f\colon A\ra B$ in \peil.
\end{proposition}

\begin{proof}
Let $X$ be the set of all maps $f\colon A\ra B$ for which $Ff$ is not well defined. Suppose further that $X$ is non-empty. Then, to each $f\in X$, let $n(f)$ be the number of vertices in the cograph $G_B$. Finally, let $N(X)=\{ n(f)~|~\forall f\in X\}$.

Since $N(X)$ is a non-empty subset of $\bn$, then by the well ordering principle, it has a least element. Choose a map $f\colon A\ra B$ corresponding to this least element. Further, suppose that the cograph for this codomain has at least one edge. If $G_B$ has no edges (i.e. $B=nW$, then we construct $Ff$ using \ref{aAnW}).

Then we have the diagram
\begin{equation*}
  \xymatrix{
A \ar@{-->}[dr]^f \ar@/_15pt/[rdd]_{f_2=\pi_2\circ f} \ar@/^15pt/[rrd]^{f_1=\pi_1\circ f} & & \\
& \ar @{} [dr] | \pbc B=(B_1\times B_2)\otimes C \ar[r]_-{\pi_1} \ar[d]^{\pi_2} & B_1\otimes C \ar[d]^{\epsilon\otimes C} \\
& B_2\otimes C \ar[r]_{\epsilon_2\otimes C} & C
}
\end{equation*}
and noting that since $G_B$ has at least one edge, then $B_1\otimes C$ and $B_2\otimes C$ each have strictly fewer vertices in their respective cographs than $G_B$. Thus, by construction, $Ff_1$ and $Ff_2$ are both well defined.

We wish to show the commutativity of
\begin{equation*}
\xymatrix{
FA \ar[r]^-{Ff_1} \ar[d]_{Ff_2} & FB_1\otimes FC \ar[d]^{\epsilon_1\otimes FC} \\
FB_2\otimes C \ar[r]_{\epsilon_2\otimes FC} & FC
}
\end{equation*}
so that $Ff$ can indeed be induced by the corresponding pullback square in $\cd$.

Let $f_3=(\epsilon_1\otimes C)\circ f_1\colon A\ra C$ in \peil, i.e. the composite
\begin{equation*}
\xymatrix{
A \ar[rd]_{f_3} \ar[r]^-{f_1} & B_1\otimes C \ar[d]^{\epsilon_1\otimes C} \\
& C
}.
\end{equation*}

Since $G_C$ has strictly fewer vertices than $G_B$, then $Ff_3$ is also well defined. But using the ideas from \ref{decompC} through to \ref{decompA}, then we have $Ff_3=(\epsilon_1\otimes FC)\circ Ff_1$. Similarly, we have $Ff_3=(\epsilon_2\otimes FC)\circ Ff_2$

But this is precisely the commutativity of
\begin{equation*}
\xymatrix{
FA \ar[r]^-{Ff_1} \ar[d]_{Ff_2} & FB_1\otimes FC \ar[d]^{\epsilon_1\otimes FC} \\
FB_2\otimes C \ar[r]_{\epsilon_2\otimes FC} & FC
}
\end{equation*}

Therefore $Ff$ is well defined. Then our original assumption must be incorrect, i.e. $X$ is an empty set.

Therefore $F$ is well defined on all maps.
\end{proof}

This now means that $F$ is not just an assignation, but rather functorial. As such, we have now shown that given a tangent structure $\bt$ on $\cm$, we can construct a strong monoidal functor $F$ that ``picks out" this tangent structure in a suitable up-to-isomorphism sense.
\end{proof}

\subsection{The Functor $F$ and the universality of \peil}\textcolor{white}{That's all, folks!}

We showed above that to equip a category \cm with a tangent structure $\mathbb{T}$ is equivalent to giving (up to a suitable isomorphism) a strong monoidal functor $F\colon \peil\ra \super(\cm)$ satisfying some extra properties.

As such, \peil becomes an initial tangent structure in the sense that it characterises any tangent structure $\bt$ via this functor $F$.

We also note that this functor $F$ only required that $\super(\cm)$ was a monoidal category (with respect to composition and with unit $1_\cm$) and that certain pullbacks were preserved. As a result, we make the following generalisation.

\begin{definition}\label{broaddef} Let $(\cg,\square,\textbf{1})$ be a monoidal category. Regard the category \peil\ as monoidal with respect to coproduct and having unit $k$. A \textit{tangent structure} \bg in \cg is a strong monoidal functor $F\colon (\peil,\otimes,k)\ra(\cg,\square,\textbf{1})$ satisfying the following conditions:

1) $F$ preserves foundational pullbacks

2) The equaliser
\begin{equation*}
  \xymatrix@C+1em@R+1em{
W^2 \ar[r]^{v} & 2W \ar@<.5ex>[rr]^{W\otimes\epsilon_W} \ar@<-.5ex>[rr]_{\eta_W\circ (\epsilon_W\otimes\epsilon_W)} & & W
}
\end{equation*}
is preserved
\end{definition}

\begin{theorem}
A tangent structure on $\cm$ (in the sense of \textnormal{\textbf{Theorem \ref{bigtheorem}}}) is the same as a tangent structure in \en (in the sense of \textnormal{\textbf{Definition \ref{broaddef}}}).
\end{theorem}

%

\end{document}